\documentclass[11pt]{amsart}
\usepackage{amsmath,amscd,amssymb}
\newtheorem{theorem}{Theorem}[section]
\newtheorem{corollary}[theorem]{Corollary}
\newtheorem{proposition}[theorem]{Proposition}

\newtheorem{lemma}[theorem]{Lemma}
\theoremstyle{definition}    

\theoremstyle{remark}

\newtheorem{remark}[theorem]{Remark}

\newtheorem{example}[theorem]{Example}
\newcommand{\pt}{\operatorname{pt}}

\newcommand\A{\mathcal{A}}

\newcommand{\V}{\mathcal{V}}
\renewcommand{\L}{\mathcal{L}}
\renewcommand{\O}{\mathcal{O}}

\newcommand{\Co}{\mathcal{C}}

\newcommand{\U}{\on{U}}

\newcommand{\R}{\mathbb{R}}
\newcommand{\C}{\mathbb{C}}

\newcommand{\Z}{\mathbb{Z}}

\newcommand\lie[1]{\mathfrak{#1}}
\renewcommand{\k}{\lie{k}}
\newcommand{\h}{\lie{h}}
\newcommand{\g}{\lie{g}}

\renewcommand{\a}{\mathsf{a}}

\newcommand{\on}{\operatorname}

\newcommand{\Ad}{ \on{Ad} }

\renewcommand{\ker}{ \on{ker}}

\newcommand{\Cour}[1]      {[\![#1]\!]}

\newcommand\qu{/\kern-.7ex/} % Categorical quotients

\renewcommand{\d}{{\mbox{d}}}

\newcommand\sig{\sigma}

\newcommand\Om{\Omega}
\newcommand\om{\omega}

\renewcommand{\deg}{\on{deg}}
\newcommand{\f}{\frac}

\newcommand{\p}{\partial}
\renewcommand{\l}{\langle}
\renewcommand{\r}{\rangle}
\newcommand\hh{{\f{1}{2}}}
\newcommand{\ti}{\tilde}
\newcommand{\eeq}{\end{eqnarray*}}
\newcommand{\beq}{\begin{eqnarray*}}

\newcommand{\pr}{\on{pr}}
\newcommand{\wh}{\widehat}

\newcommand{\mf}{\mathfrak}
\begin{document}
%\sloppy

\title[]{The Atiyah algebroid\\ of the path fibration over
a Lie group}

\author{A. Alekseev}
\address{University of Geneva, Section of Mathematics,
2-4 rue du Li\`evre, c.p. 64, 1211 Gen\`eve 4, Switzerland}
\email{Anton.Alekseev@unige.ch}

\author{E. Meinrenken}
\address{University of Toronto, Department of Mathematics,
40 St George Street, Toronto, Ontario M4S2E4, Canada }
\email{mein@math.toronto.edu}

\begin{abstract}
  Let $G$ be a connected Lie group, $LG$ its loop group, and
  $\pi\colon PG\to G$ the principal $LG$-bundle defined by
  quasi-periodic paths in $G$.  This paper is devoted to differential
  geometry of the Atiyah algebroid $A=T(PG)/LG$ of this bundle. Given
  a symmetric bilinear form on $\g$ and the corresponding central
  extension of $L\g$, we consider the lifting problem for $A$, and
  show how the cohomology class of the Cartan 3-form $\eta\in\Om^3(G)$
  arises as an obstruction. This involves the construction of a 2-form
  $\varpi\in \Om^2(PG)^{LG}=\Gamma(\wedge^2 A^*)$ with
  $\d\varpi=\pi^*\eta$. In the second part of this paper we obtain
  similar $LG$-invariant primitives for the higher degree analogues of the
  form $\eta$, and for their $G$-equivariant extensions.
\end{abstract}
\maketitle
\setcounter{tocdepth}{3}
\tableofcontents
\section{Introduction}
Let $G$ be a connected Lie group with loop group $LG$. Denote by
$\pi\colon PG\to G$
the principal $LG$-bundle, given by the set of `quasi-periodic' paths in $G$. 
Thus $\gamma\in C^\infty(\R,G)$ belongs to the fiber $(PG)_g$ if 
it has the property,
\[ \gamma(t+1)=g\gamma(t)\]
for all $t$. The principal action of $\lambda\in LG$ reads 
$(\lambda.\gamma)(t)=\gamma(t)\lambda(t)^{-1}$; it commutes with the
action of $a\in G$ given as $(a.\gamma)(t)=a\gamma(t)$. 

We are interested in the differential geometry of the
infinite-dimensional space $PG\to G$. Since all of our considerations
will be $LG$-equivariant, it is convenient to phrase this discussion
in terms of the \emph{Atiyah algebroid} $A=T(PG)/LG\to G$. As
explained below, the fiber of $A$ at $g\in G$ consists of paths
$\xi\in C^\infty(\R,\g)$ such that $\xi(t+1)-\Ad_g\xi(t)=:v_\xi$ is constant.
We may directly write down the Lie algebroid bracket on sections
of $A$, thus avoiding a discussion of Lie brackets of vector fields on
infinite dimensional spaces. The Lie algebra bundle $L\subset A$,
given as the kernel of the anchor map, has fibers the \emph{twisted
  loop algebras} defined by the condition $\xi(t+1)=\Ad_g\xi(t)$.

An invariant symmetric bilinear form on $\g$ defines a central
extension $\wh{L}\to L$ by the trivial bundle $G\times\R$. One may
then ask for a lift $\wh{A}\to A$ of the Atiyah algebroid to this
central extension. More generally, we will study a similar lifting
problem for any transitive Lie algebroid $A$ over a manifold $M$.
We will show that the choice of a connection on $A$, together with a
`splitting', define an element $\varpi\in \Gamma(\wedge^2 A^*)$ whose
Lie algebroid differential is basic. The latter defines a closed
3-form $\eta\in\Om^3(M)$, whose cohomology class turns out to be the
obstruction to the lifting problem. In the case of the Atiyah
algebroid over $G$, with suitable choice of connection, $\eta$ is the
Cartan 3-form, while $\varpi$ is explicitly given as 
\[ \varpi(\xi,\zeta)=
\int_0^1 \dot\xi\cdot \zeta-\hh v_\xi\cdot v_\zeta-\Ad_g\xi(0)\cdot
v_\zeta
\]
for $\xi,\zeta\in\Gamma(A)$. Similarly, the obstruction for the
$G$-equivariant lifting problem is the equivariant Cartan 3-form
$\eta_G$, while it turns out that $\varpi_G=\varpi$. Note that 
$\varpi$ may be viewed as a $G\times LG$-invariant 2-form on $PG$. 

The second part of this paper is devoted to `higher analogues' of the 
equations $\d\varpi=\pi^*\eta$, respectively $\d_G\varpi=\pi^*\eta_G$. For any 
invariant polynomial $p\in (S\g^*)^G$ of homogeneous degree $k$, 
the Bott-Shulman construction \cite{bo:le,je:gr,shu:th} defines closed forms 
\[ \eta^p\in \Om^{2d-1}(G),\ \ \eta^p_G\in\Om^{2d-1}_G(G).\] 
These become exact if pulled back to elements of $ \Gamma(\wedge A^*)$, 
and we will construct explicit primitives 
\[ \varpi^p\in \Gamma(\wedge^{2d-2} A^*),\ \ \varpi^p_G\in
\Gamma_G(\wedge^{2d-2} A^*).\]
which may be viewed as $LG$-invariant differential forms on $PG$. We
stress that while the existence of primitives of $\pi^*\eta^p,\ 
\pi^*\eta^p_G$ is fairly obvious, the existence of
\emph{$LG$-invariant} primitives is less evident. Pulling $\varpi^p$
back to the fiber over the identity $LG=(PG)_e$, one recovers the
closed invariant forms on the loop group $LG$ discussed in
Pressley-Segal \cite{pr:lo}.

\vskip 0.2cm

{\bf Acknowledgments:}  Research of A.A. was supported in part by 
the grants 200020-120042 and 200020-121675 of the Swiss National Science Foundation.
E.M. was supported by an NSERC Discovery Grant and a Steacie Fellowship.

\section{Review of transitive Lie algebroids}\label{sec:trans}
In this Sections we collect some basic facts about connections and
curvature on transitive Lie algebroids. Most of this material is due
to Mackenzie, and we refer to his book \cite{mac:gen} or to the
lecture notes by Crainic-Fernandes \cite{cra:lect} for further
details. 

\subsection{Lie algebroids}
A Lie algebroid is a smooth vector bundle $A\to N$, with a Lie bracket
on the space of sections $\Gamma(A)$ and an \emph{anchor map}
$\a\colon A\to TN$ satisfying the Leibniz rule,
$[\xi_1,f\xi_2]_A=f[\xi_1,\xi_2]_A+\a(\xi_1)(f)\,\xi_2$. This implies
that $\a$ induces a Lie algebra homomorphism on sections. An example
of a Lie algebroid is the Atiyah algebroid $TP/H$ of a principal
$H$-bundle $P\to N$, where $\Gamma(TP/H)=\mf{X}(P)^H$ with the usual
bracket of vector fields. A \emph{representation} of a Lie algebroid
$A$ on a vector bundle $\V\to N$ is given by a \emph{flat
  $A$-connection} on $\V$, i.e. by a $C^\infty(N)$-linear Lie algebra
homomorphism $\Gamma(A)\to \on{End}(\Gamma(\V)),\ \xi\mapsto
\nabla_\xi$ satisfying the Leibnitz
rule, $\nabla_\xi(f\sig)=f\nabla_\xi\sig+\a(\xi)(f)\sig$. Given additional structure on $\V$
one can ask for the representation to preserve that structure: For
instance, if $\V=L$ is a bundle of Lie algebras, one would impose that
$\nabla_\xi$ acts by derivations of the bracket $[\cdot,\cdot]_L$.
Tensor products and direct sums of $A$-representations are defined in
the obvious way. The \emph{trivial} $A$-representation is the bundle
$\V=N\times\R$ with $\nabla_\xi=\a(\xi)$ given by the anchor map.

Suppose $\V$ is an $A$-representation, and consider the graded
$\Gamma(\wedge A^*)$-module $\Gamma(\wedge A^*\otimes\V)$.
Generalizing from $A=TM$, we will think of the sections of $\wedge A^*
\otimes \V$ as \emph{$\V$-valued forms} on $A$. 
For $\xi\in\Gamma(A)$, the \emph{Lie derivatives} $\L_\xi$ are the
operators of degree $0$ on $\Gamma(\wedge A^*\otimes\V)$,
defined inductively by 
\[ \iota_\zeta\circ
\L_\xi=\L_\xi\circ\iota_\zeta-\iota_{[\xi,\zeta]_A},\ \ \ \zeta\in
\Gamma(A),\ \]
with $\L_\xi\sig=\nabla_\xi\sig$ for $\sig\in \Gamma(\V)$. Here
$\iota_\xi$ are the operators of \emph{contraction} by $\xi$. 
Similarly, $\d$ is the operator of degree $1$ on $\Gamma(\wedge A^*\otimes\V)$
defined by Cartan's identity $\iota_\xi\circ \d=\L_\xi-\d\circ\iota_\xi$. The operators
$\iota_\xi,\L_\xi,\d$ satisfy the usual commutation relations of
contractions, Lie derivative and differential. In particular, $\d$
squares to zero.

%In what follows, we will think of $\Gamma(\V)$ as the degree $0$ piece
%of $\Gamma(\wedge A^*\otimes\V)$, and directly write $\L_\xi\sig$ in
%place of $\nabla_\xi\sig$.

\subsection{Transitive Lie algebroids}
A Lie algebroid $A$ over $N$ is called \emph{transitive} if its anchor
map $\a\colon A\to TN$ is surjective. In that case, the kernel of the
anchor map is a bundle $L\to N$ of Lie algebras, 
and we have the exact sequence of Lie algebroids,
\begin{equation}\label{eq:transalg} 
0\to L\to A\to TN\to 0.\end{equation}
The \emph{structure Lie algebra bundle} $L$ carries an
$A$-representation $\nabla_\xi\zeta=[\xi,\zeta]$ ($\xi\in\Gamma(A),\ 
\zeta\in\Gamma(L)$) by derivations of the Lie bracket.
\begin{example}\label{ex:prin}
  The Atiyah algebroid $A=TP/H$ of a principal bundle is a transitive
  Lie algebroid, with $L$ the associated bundle of Lie algebras
  $L=P\times_H \h$. The bracket on $\Gamma(A)$ is given by its
  identification with $H$-invariant vector fields on $P$. The induced
  bracket on $\Gamma(L)$ is \emph{minus} the pointwise bracket on
  $C^\infty(P,\h)^H\cong \Gamma(L)$.
\end{example}

The dual $\a^*\colon T^*N\to
A^*$ of the anchor map extends to the exterior algebras. 
Given an $A$-representation $\V$,  it hence gives
an injective map $\a^*\colon \wedge T^*N\otimes\V\to \wedge A^*\otimes
\V$, defining a map on sections, 
\[ \a^*\colon \Om(N,\V)\cong \Gamma(\wedge T^*N\otimes \V)\to \Gamma(\wedge A^*\otimes
\V).\]
The image of this map is the \emph{horizontal subspace} $\Gamma(\wedge
A^*\otimes \V)_{\on{hor}}$, consisting of sections $\phi$ satisfying
$\iota_\xi\phi=0$ for all $\xi\in \Gamma(L)$. We will often view
$\Om(N,\V)$ as a subspace of $\Gamma(\wedge A^*\otimes \V)$, without
always spelling out the inclusion map $\a^*$. The \emph{basic
  subcomplex} $\Gamma(\wedge A^*\otimes \V)_{\on{basic}}$ is the
subspace of horizontal sections satisfying $\L_\xi\phi=0$ for all
$\xi\in \Gamma(L)$; it is stable under the
differential $\d$.
\begin{lemma}  \label{lemma}
Suppose that the $A$-connection on $\V$ descends to a flat $TN=A/L$-connection, 
i.e. that $\nabla_\xi=0$ for $\xi\in \Gamma(L)$. Then 
\[\Gamma(\wedge A^*,\V)_{\on{basic}}\cong  \Gamma(\wedge
A^*,\V)_{\on{hor}}\cong \Om(N,\V).\] 
\end{lemma}
\begin{proof}
  Let $\xi\in \Gamma(L)$ so that $\a(\xi)=0$, and let $\phi\in
  \Gamma(\wedge^k A^*,\V)_{\on{hor}}$. We will show $\L_\xi\phi=0$ by
  induction on $k$. If $k=0$, we have $\L_\xi\phi=\nabla_\xi\phi=0$ by
  assumption. If $k>0$, the induction hypothesis shows that for all $\zeta\in
  \Gamma(A)$, $\iota_\zeta \L_\xi\phi=\L_\xi
  \iota_\zeta\phi-\iota_{[\xi,\zeta]_A}\phi=0$, hence $\L_\xi\phi=0$.
  Here we used that $\Gamma(\wedge A^*,\V)_{\on{hor}}$ is stable under
  $\iota_\zeta$ and that $[\xi,\zeta]_A\in \Gamma(L)$.
\end{proof}

\begin{remark}
  Lemma \ref{lemma} applies in particular to the trivial $A$-representation
  $\V=N\times\R$. Thus $\Gamma(\wedge A^*)_{\on{basic}}\cong \Om(N)$.
  For general $A$-representations the space
  $\Gamma(\wedge A^*\otimes\V)_{\on{basic}}$ can be strictly smaller
  than $\Om(N,\V)$.  For instance, if $N=\pt$, so that $A=\k$ is a Lie
  algebra and $\V=V$ is a $\k$-representation, the space
  $\Gamma(\wedge A^*\otimes\V)_{\on{basic}}=V^\k$ is the space of
  $\k$-invariants, while $\Om(N,\V)=V$.
\end{remark}

A \emph{connection} on a transitive Lie algebroid is a left
splitting $\theta\colon A\to L$ of the exact sequence
\eqref{eq:transalg}. The corresponding right splitting
$\on{Hor}^\theta\colon TN\to A$ is called the \emph{horizontal lift}.
Dually, the connection defines a horizontal projection 
\[ \on{Hor}^\theta_*\colon \Gamma(\wedge A^*\otimes \V)\to
\Gamma(\wedge A^*\otimes \V)_{\on{hor}}.\]
One defines the \emph{covariant derivative} by
$\d^\theta=\on{Hor}^\theta_*\circ \d$, and the  
\emph{curvature} of $\theta$ is given as
\footnote{The minus sign in this formula is consistent with Example
  \ref{ex:prin}.}
\begin{equation}\label{eq:curv}
 F^\theta=\d^\theta\theta=\d\theta-\hh
 [\theta,\theta]_A\in 
\Gamma(\wedge^2 A^*\otimes
 L)_{\on{hor}}
\end{equation}
%
%That is, 
%
%\[\begin{split}
%F^\theta(\xi_1,\xi_2)&=-\theta([\xi_1,\xi_2]_A)+[\xi_1,(\theta(\xi_2))]_A
%-[\xi_2,\theta(\xi_1)]_A-[\theta(\xi_1),\theta(\xi_2)]_A\\
%&=[\xi_1-\theta(\xi_1),\xi_2-\theta(\xi_2)]_A-([\xi_1,\xi_2]_A-\theta([\xi_1,\xi_2]_A)
%\end{split}\]
%
%In terms of the decomposition $A=L\oplus TN$ defined by the connection
%$\theta$, the vertical part of the bracket on $\Gamma(A)$ is given by 
%
%\[\theta([\xi,\zeta]_A)=\L_\xi(\theta(\zeta))-\L_\zeta(\theta(\xi))-[\theta(\xi),\theta(\zeta)]_L
%+F^\theta(a(\xi),a(\zeta)).\]
%
%where $\L_{\cdot}$ are the Lie derivatives for the action of $A$ on
%$L$. (The horizontal part of the bracket is determined by bracket on
%$\mf{X}(N)$, since $\Gamma(A)\to \mf{X}(N)$ is a Lie algebra homomorphism.)
%
%On $\Gamma(\wedge A^*\otimes
%\V)_{\on{basic}}$ it is given by the formula
%$\d^\theta=\d-\varrho(\theta)$, where $\varrho\colon \wedge
%A^*\otimes A\to \on{End}(\Gamma(\wedge A^*\otimes\V))$ 
%is the extension of the Lie derivatives. One has, 
%$d^\theta\circ \d^\theta=\varrho(F^\theta)$.

\subsection{Pull-backs}
We recall the notion of \emph{pull-back Lie algebroids}, due to
Higgins-Mackenzie \cite{hig:alg}, for the special case of transitive Lie algebroids.
Suppose $A\to N$ is a transitive Lie algebroid, and $\Phi\colon M\to
N$ is a smooth map. Let $\Phi^! A\to M$ be the bundle \footnote{We
  remark that our use of the notation $\Phi^!$ is different from that
  in the book \cite{mac:gen}.}  over $M$, defined by the fiber product
diagram
\[ \begin{CD} \Phi^! A @>>> A\\
@VVV @VV{a}V\\
TM @>>{d\Phi}> TN
\end{CD}\]
That is, $\Phi^!A=(d\Phi)^*A$ if $A$ is viewed as a bundle over $TN$. 
Then $\Phi^! A$ carries a natural structure of a transitive Lie
algebroid, with the left vertical map $\Phi^! A\to TM$ as the anchor
map, and the upper horizontal map is a morphism of Lie algebroids.  

We refer to $\Phi^! A$ as the pull-back of $A$ by the map $\Phi$.  It
is a pull-back in the category of Lie algebroids, not to be confused
with the pull-back $\Phi^* A$ of $A$ as a vector bundle. For instance,
taking $A=TN$ one has $\Phi^! TN=TM\not= \Phi^*TN$.  Note that if
$A=TP/H$ is the Atiyah algebroid of a principal $H$-bundle $P\to N$,
then $\Phi^! A=T(\Phi^*P)/H$ is the Atiyah algebroid of the pull-back
principal bundle.

The kernel of the anchor map of $\Phi^! A$ is $\Phi^*L$, the usual
pull-back as a bundle of Lie algebras. For any $A$-representation
$\V$, the pull-back $\Phi^*\V$ inherits a $\Phi^! A$-representation,
and there is a natural cochain map $\Phi^!\colon \Gamma(\wedge
A^*\otimes\V) \to \Gamma(\wedge(\Phi^! A)\otimes\Phi^*\V)$.  Given a
connection $\theta\colon A\to L$, the pull-back algebroid inherits a
pull-back connection $\Phi^! \theta\colon \Phi^! A\to \Phi^*L$. The
curvature of the pull-back connection is $F^{\Phi^!\theta}=\Phi^!
F^\theta$.

\subsection{Equivariant transitive Lie algebroids}\label{subsec:equiv}
Suppose $G$ is a Lie group acting on $A\to N$ by Lie algebroid
homomorphisms. By \emph{infinitesimal generators} for the action we
mean a $G$-equivariant map
\begin{equation}\label{eq:infaction}
 \g\to \Gamma(A),\ x\mapsto x_A\end{equation}
with the property $[x_A,\xi]=\f{\p}{\p u}\Big|_{u=0}\exp(u x).\xi$.
It is then automatic that \eqref{eq:infaction} is a Lie algebra
homomorphism.
For any $G$-equivariant $A$-representation $\V$, the complex $\wedge
A^*\otimes\V$ becomes a $G$-differential space (cf.
\cite{gu:sy,bru:eq}), with contraction operators
$\iota_x=\iota_{x_A}$. One may hence introduce the
equivariant complex
\[ \Gamma_G(\wedge A^*\otimes\V):=(S\g^*\otimes \Gamma(\wedge A^*\otimes\V))^{G}\]
with differential $\d_G=1\otimes \d-\sum_j e^j\otimes  \iota_{e_j}$ for a basis $e_j$ of $\g$,
with dual basis $e^j$ of $\g^*$. For $A=TN$ this complex is denoted
$\Om_G(N,\V)$. Replacing $\d$ with $\d_G$ in the discussion above, one
may introduce equivariant curvatures $F^\theta_G$ for $G$-invariant connections on
$A$:
\[ F^\theta_G=\d_G\theta- \hh [\theta,\theta]_A=F^\theta-\Psi\in 
\Gamma_G(\wedge^2 A^*\otimes L)_{\on{hor}}.\] 
Here $\Psi(x)=\iota_x\theta\in \Gamma(\wedge^0 A^*\otimes L)$ for $x\in\g$.

\section{A lifting problem for transitive Lie algebroids}\label{sec:lifting}
Let $H$ be a Lie group, and $\pi\colon P\to N$ a smooth principal
$H$-bundle. Given a central extension $\wh{H}\to H$ of the structure
group by $\U(1)$, it is not always possible to lift $P$ to a principal
$\wh{H}$-bundle. As is well-known, the obstruction class is an element
of $H^3(N,\Z)$. A construction of Brylinski \cite{br:lo} gives an
explicit de Rham representative of the image of this class in
$H^3(N,\R)$. In this Section, we will develop the analogue of
Brylinski's theory for transitive Lie algebroids.
\subsection{The lifting problem}
Let $A\to N$ be a transitive Lie algebroid with anchor map $\a\colon
A\to TN$, and with structure Lie algebra bundle $L=\on{ker}(\a)$. Suppose
that
\begin{equation}\label{eq:lieseq} 
0\to N\times \R \to \wh{L}\xrightarrow{p} L\to 0\end{equation}
is a central extension, where $\wh{L}$ carries an $A$-representation
(by derivations of the Lie bracket on sections), lifting that on $L$.
The lifting problem is to find a central extension of Lie algebroids
\begin{equation}\label{eq:lift} 0\to N\times \R \to \wh{A}\xrightarrow{p} A\to 0\end{equation}
such that $\wh{L}$ is realized the kernel of the anchor map $\wh{A}\to
TN$. We may also consider the lifting problem for a given connection
$\theta\colon A\to L$, where we declare that $(\wh{A},\wh{\theta})$ lifts $(A,\theta)$
if $\wh{A}$ lifts $A$ and $p\circ \wh{\theta}=\theta\circ p$.
\begin{example}[Principal bundles I]
\label{ex:princ}
In the principal bundle case, $A=TP/H$ is the Atiyah algebroid,
$L=P\times_H\h$, and one obtains a lifting problem $\wh{L}=P\times_H
\wh{\h}$ for any given central extension $0\to \R\to \wh{\h}\to \h\to
0$ of Lie algebras. Suppose these integrate to an exact sequence $1\to
\U(1)\to \wh{H}\to H\to 1$ on the group level. Then for any principal
$\wh{H}$-bundle $\wh{P}$ lifting $P$, its Atiyah algebroid
$\wh{A}=T\wh{P}/\wh{H}$ is a lift of $A$ in the above sense.
\end{example}
\subsection{Splittings}
The set of splittings $j\colon L\to \wh{L}$ of the exact sequence
\eqref{eq:lieseq} is an affine space, with underlying vector space
$\Gamma(L^*)$. Any splitting $j$ determines a cocycle $\sig\in
\Gamma(\wedge^2 L^*)$, where
\[ \sig(\xi,\zeta)=j([\xi,\zeta]_L)-[j(\xi),j(\zeta)]_{\wh{L}},\ \ \ 
\xi,\zeta\in\Gamma(L).\]
(The right hand side lies in the kernel of $p$, hence it
takes values in the trivial bundle $N\times\R\subset \wh{L}$.) The
bracket on $\wh{L}$ is given in terms of this cocycle as
\begin{equation}\label{eq:lbracket}
 [\hat\xi,\hat\zeta]_{\hat L}=j([\xi,\zeta]_L)-\sig(\xi,\zeta) 
\end{equation}
where $\xi=p(\hat\xi),\ \zeta=p(\hat\zeta)$. Let $\theta\in
\Gamma(A^*\otimes L)$ be a principal connection, and consider the
covariant derivative of $j\in\Gamma(\wedge^0 A^*\otimes(L^*\otimes\wh{L}))$.
\begin{proposition}\label{prop:dj}
Both $\d j$ and $\d^\theta j$ map to $0$ under $p$. Thus
\[ \d j \in \Gamma(\wedge^1
A^*\otimes L^*),\ \ \d^\theta j\in\Gamma(\wedge^1
A^*\otimes L^*)_{\on{hor}}.\] 
One has $\d^\theta j=\d j+\sig(\theta,\cdot)$. The differential of $\sig$ 
is related to the differential of $j$ by
\[ (\d\sig)(\xi_1,\xi_2)=\l d j,\ [\xi_1,\xi_2]_L\r\] 
for $\xi_1,\xi_2\in \Gamma(L)$. 
\end{proposition}
\begin{proof}
  The first claim follows since $p(j)=\on{id}_L\in L^*\otimes L$,
  hence $p(\d j)=\d p(j)=0$. To prove the formula for $\d^\theta j$ we
  compute, for $\xi\in \Gamma(L)$ and $\zeta\in \Gamma(A)$,
\[ \iota_\zeta \l \d j,\xi\r=\l \L_\zeta
j,\xi\r=\L_\zeta(j(\xi))-j(\L_\zeta\xi).\]
For $\zeta \in \Gamma(L)$ the right hand side is equal to $-\sig(\zeta,\xi)$, 
and we obtain $\iota_\zeta(\d j+\sig(\theta,\cdot))=0$. This shows that $\d
j+\sig(\theta,\cdot)$ is horizontal. On the other hand, it is obvious
that $\d j$ and $\d j+\sig(\theta,\cdot)$ agree on horizontal vectors.
Now let $\xi_1,\xi_2\in \Gamma(L)$ and $\zeta\in \Gamma(A)$. We
compute, using $(\L_\zeta j)(\xi_i)=(\iota_\zeta \d j)(\xi_i)\in
C^\infty(N)\subset \Gamma(\wh{L})$, 
\[
\begin{split}
\iota_\zeta (\d \sigma)(\xi_1,\xi_2)&=(\L_\zeta \sigma)(\xi_1, \xi_2)\\
& = \L_\zeta (\sigma(\xi_1,\xi_2)) - \sigma(\L_\zeta \xi_1, \xi_2) - \sigma(\xi_1, L_\zeta \xi_2) \\
& = \L_\zeta \big(j([\xi_1, \xi_2]_L) - [j(\xi_1), j(\xi_2)]_{\wh{L}}\big) 
-j([\L_\zeta \xi_1, \xi_2]_L) +[j(\L_\zeta \xi_1), j(\xi_2)]_{\wh{L}} \\&\ \ \ - 
j([\xi_1, \L_\zeta \xi_2]_L) + [j( \xi_1), j(\L_\zeta \xi_2)]_{\wh{L}} \\
& = (\L_\zeta j)([\xi_1, \xi_2]_L) - [(\L_\zeta j)(\xi_1), j(\xi_2)]_{\wh{L}} -
[j(\xi_1), (\L_\zeta j)(\xi_2)]_{\wh{L}} \\
& = (\L_\zeta j)([\xi_1, \xi_2]_L) = \iota_\zeta (\d j)([\xi_1, \xi_2]_L).
\end{split}
\]
Hence $(\d \sigma)(\xi_1,\xi_2)=(\d j)([\xi_1, \xi_2]_L)$. 
\end{proof}

\subsection{The form $\varpi$}
Let $F^{j(\theta)}\in \Gamma(\wedge^2 A^*\otimes \wh{L})$
be the curvature-like expression,
\[ F^{j(\theta)}=\d (j(\theta))-\hh [j(\theta),j(\theta)]_{\wh{L}}.\] 
%
%where we have extended $j$ to $\Gamma(\wedge A^*\otimes L)$. 
Since $p(F^{j(\theta)})=F^\theta$ the difference
\[ \varpi:=F^{j(\theta)}-j(F^\theta)\] 
is scalar-valued, i.e. it is an element of $\Gamma(\wedge^2 A^*)$.
\begin{proposition}
The 2-form $\varpi\in\Gamma(\wedge^2 A^*)$ is given by the formula, 
\begin{equation}\label{eq:brylinski} 
\varpi=\l \d j, \theta \r + \hh \sig(\theta,\theta)= \l \d^\theta j,\theta\r -\hh
  \sig(\theta,\theta)
.\end{equation}
Its differential is basic, so that $\d\varpi=\a^*\eta$ for a 
closed 3-form $\eta\in
\Om^3(N)$. We have 
\begin{equation}\label{eq:eta}
\a^*\eta=-\l \d^\theta j,F^\theta\r.
\end{equation}
The contractions of $\varpi$ with $\xi\in\Gamma(L)$ are given by 
$\iota_\xi\varpi=-\l\d j,\xi\r$.  
\end{proposition}
\begin{proof}
We compute:
\[ \begin{split}
\varpi
&=\d (j(\theta))-j(\d\theta)-\hh\big([j(\theta),j(\theta)]_{\wh{L}}-j([\theta,\theta]_L)\big)\\
&=\l \d j,\theta\r+\hh \sig(\theta,\theta)\\
&=\l \d^\theta j,\theta\r -\hh
  \sig(\theta,\theta),\\
\d\varpi&=-\l \d j,\d\theta\r-\sig(\theta,\d\theta)+\hh (\d\sigma)(\theta,\theta)\\
&=-\l \d^\theta j,\d\theta\r+\hh
(\d\sigma)(\theta,\theta)\\
&=-\l \d^\theta j,F^\theta\r.
\end{split}\]
Here we have used $\sig(\theta,[\theta,\theta])=0$ and Proposition
\ref{prop:dj}. Since $\d \varpi \in \Gamma(\wedge^3 A^*)$ is horizontal,
by Lemma \ref{lemma} it is also basic. Hence, it is an image 
of a unique closed 3-form $\eta \in \Omega^3(N)$ under the map $\a^*$.

Finally, for the contractions of $\varpi$ with elements
$\xi\in \Gamma(L)$ we find,
\[ \iota_\xi\varpi=-\l \d^\theta j,\xi \r+\sig(\theta,\xi)=-\l \d j,\xi\r. \] 
\end{proof}
The next Proposition describes the dependence of $\eta$ on
the choice of splitting and connection. 
\begin{proposition}
  Let $j'=j+\beta$ be a new splitting, where $\beta\in \Gamma(L^*)$,
  and $\theta'=\theta+\lambda$ a new connection, with $\lambda\in
  \Gamma(A^*\otimes L)_{\on{hor}}$. Then $\eta'-\eta=\d\gamma$ where $\gamma\in\Om^2(N)$ is
  given by the following element of $\Gamma(\wedge^2
  A^*)_{\on{basic}}$, 
\[ \a^*\gamma=\l \d^\theta
j,\lambda\r+ \hh
\sig(\lambda,\lambda)-\l\beta,F^\theta+\d^\theta\lambda\r
+ \hh \beta([\lambda, \lambda]_L).\]
In particular, the cohomology class $[\eta]\in H^3(N,\R)$ is
independent of the choices of $j,\theta$. 
\end{proposition}
\begin{proof}
  From its defining formula, we see that the cocycle $\sigma$ changes
  by a coboundary:
$\sig'(\xi_1,\xi_2)=\sig(\xi_1,\xi_2)+\beta([\xi_1,\xi_2]_L)$. 
Hence, 
\[ \varpi'= \l \d j', \theta' \r + \hh \sigma(\theta', \theta') +
\hh \beta([\theta', \theta']_L) . \]
First, consider terms which do not involve $\beta$:
\[ \l \d j, \theta+\lambda \r + \hh \sigma(\theta+\lambda,\theta+\lambda) =
\varpi + \l \d^\theta j, \lambda \r + \hh \sigma(\lambda, \lambda) . \]
The remaining terms may be written as 
\[ \begin{split}
  \l \d \beta, \theta+\lambda \r + 
  \hh \beta([\theta+\lambda, \theta+\lambda]_L)
  &= \d \l \beta, \theta+\lambda \r - \l \beta, \d \theta + \d \lambda \r
  + \hh \beta([\theta+\lambda, \theta+\lambda]_L) 
  \\&=
  \d \l\beta,\theta+\lambda\r-\l\beta,F^\theta+\d^\theta\lambda\r
  + \hh \beta([\lambda, \lambda]_L) . \\
\end{split}\]
Hence, $\d(\varpi'-\varpi)=\d \sigma$, where
\[ \sigma=\l \d^\theta j,\lambda\r+ \hh
\sig(\lambda,\lambda)-\l\beta,F^\theta+\d^\theta\lambda\r
+ \hh \beta([\lambda, \lambda]_L).\]
Since $\sigma \in \Gamma(\wedge^2 A^*)_{\on{hor}}$, by Lemma \ref{lemma}
it is basic, and we conclude that $\gamma \in \Omega^2(N)$ defined
by equality $\a^*\gamma =\sigma$ satisfies $\eta'-\eta=\d \gamma$.
\end{proof}
\subsection{The cohomology class $[\eta]$ as an obstruction class}
We will now show that the cohomology class of $\eta$ is precisely the
obstruction class for our lifting problem.
\begin{theorem}\label{th:liftingproblem}
    Suppose that $\theta$ is a connection on $A$ and 
  $j\colon L\to \wh{L}$ is a splitting. Then there is a 1-1-
  correspondence between:
\begin{enumerate}
\item isomorphism classes of lifts $(\wh{A},\wh{\theta})$ of the data $(A,\theta)$, and 
\item 2-forms $\om\in\Om^2(N)$ such that $\d\om=-\eta$.
\end{enumerate}
It follows that $[\eta]=0$ precisely if the lifting problem
\eqref{eq:lieseq}, \eqref{eq:lift} admits a solution.
\end{theorem}
\begin{proof}
We first show how to construct a solution of the lifting problem, provided
$\eta$ is exact. Let $A=L\oplus TN$ be the decomposition defined by the connection $\theta$,
and put
\[ \wh{A}:=\wh{L}\oplus TN\] 
with the obvious projection $p\colon \wh{A}\to A$, with the connection
$\hat\theta$ the projection to the first summand, and with anchor map
$\wh{a}$ the projection to the second summand. Let $j_A=j\oplus
\on{id}_{TN}\colon A\to \wh{A}$.  We want to consider Lie brackets
$[\cdot,\cdot]_{\hat{A}}$ on $\Gamma(\wh{A})$, extending the bracket
on $\Gamma(\wh{L})$, and such that $p$ induces a Lie algebra
homomorphism $\Gamma(\wh{A})\to \Gamma(A)$.  If
$\hat\zeta\in\Gamma(\hat L )$, we have
\[ -[\hat\zeta,\hat\xi]_{\hat A}=[\hat\xi,\hat\zeta]_{\hat A}=\L_\xi
\hat\zeta,\ \ \xi=p(\hat\xi)\] 
using the $A$-representation on $\hat L$.  
Hence, the bracket is determined if one of the entries is a section of
$\wh{L}$. Consequently we only need to specify the bracket on horizontal
sections. For $X,Y\in\mf{X}(N)$ these brackets will have the form
\[\begin{split}
 [\on{Hor}^{\hat\theta}(X),\on{Hor}^{\hat\theta}(Y)]_{\hat A}
&=j_A([\on{Hor}^{\theta}(X),\on{Hor}^{\theta}(Y)]_{A})-\om(X,Y)\\
&=\on{Hor}^{\hat\theta}([X,Y])+j(F^\theta(X,Y))-\om(X,Y)\end{split}
\]
for some 2-form $\om\in\Om^2(N)$. Having defined the bracket in this
way, consider the Jacobi identity
\[ [\hat\xi_1,[\hat\xi_2,\hat\xi_3]_{\hat A}]_{\hat A}+\on{cycl.}=0.\]
If $\hat\xi_3$ lies in $\Gamma(\wh{L})$, this identity is equivalent to 
the representation property, 
$\L_{\xi_1}\L_{\xi_2}\hat\xi_3-\L_{\xi_2}\L_{\xi_1}\hat\xi_3-\L_{[\xi_1,\xi_2]}\hat\xi_3=0$, 
where $\xi_i=p(\hat\xi_i)$. Hence, the Jacobi identity is automatic if
one of the entries lies in $\wh{L}$. It remains to consider the case
that $\xi_i=\on{Hor}^{\hat\theta}(X_i)$ for $i=1,2,3$. Separating
terms according to the decomposition $\wh{A}=(N\times\R)\oplus A$, we
have,
\[
\begin{split}
\lefteqn{ [\on{Hor}^{\hat\theta}(X_1),[\on{Hor}^{\hat\theta}(X_2),\on{Hor}^{\hat\theta}(X_3)]_{\hat
  A}]_{\hat A}}\\
&=-\L_{X_1}\om(X_2,X_3)+\om(X_1,[X_2,X_3])
+\l d^\theta j(X_1),\ F^\theta(X_2,X_3)\r+\cdots
\end{split}\]
where $\cdots$ indicates sections of $j_A(A)\cong 0\oplus A$. 
So the scalar part of the Jacobi identity reads, 
\[ -\L_{X_1}\om(X_2,X_3)+\om(X_1,[X_2,X_3])+\l d^\theta j(X_1),\ F^\theta(X_2,X_3)\r+\on{cycl.}=0\]
Equivalently, $\d\om-\l\d^\theta j,F^\theta\r=0$. By Equation
\eqref{eq:eta}, we have $\a^*\eta=-\l\d^\theta j,F^\theta\r$. We hence
conclude that the bracket $[\cdot,\cdot]_{\hat A}$ defined by $\om$ is
a Lie bracket if and only if $\d\om=-\eta$.

Conversely, if $p\colon \wh{A}\to A$ is a solution of the lifting
problem, choose a connection $\wh{\theta}$ lifting $\theta$. This
gives a splitting $\wh{A}=\wh{L}\oplus TN$ lifting the splitting of
$A$. Define $\om$ as the scalar component of the bracket on
$\Gamma(\hat A)$, restricted to horizontal sections. The
calculation above shows that $\d\om=-\eta$. 
\end{proof}

\begin{example}[Principal bundles II]
  This is a continuation of Example \ref{ex:princ}, where we
  considered the lifting problem for a principal $H$-bundle $P\to N$.
  It was shown by Brylinski that $[\eta]\in H^3(N,\R)$ is the image of
  the obstruction class under the coefficient homomorphism
  $H^3(N,\Z)\to H^3(N,\R)$.  Given a solution $\wh{P}\to N$ of the
  lifting problem, the connection $\theta\in \Gamma(A^*\otimes L)\cong
  \Om^1(P,\h)^H$ is an ordinary principal connection on $P$, and
  $\wh{\theta}$ is a lift of $\theta$ to $\wh{P}$.
\end{example}

\subsection{The equivariant lifting problem}
\label{subsec:equivariant}
Suppose now that the sequence \eqref{eq:lieseq} is $G$-equivariant,
with the trivial action on the bundle $N\times \R$, and that the
$A$-representation on $\Gamma(\wh{L})$ is $G$-equivariant.  
We may then consider the $G$-equivariant version of the lifting problem: Thus, we
are looking for a $G$-equivariant lift $\wh{A}\to A$, such that the
action on $\wh{A}$ has infinitesimal generators $x_{\wh{A}}$
satisfying $p(x_{\wh{A}})=x_A$.

Suppose there is a $G$-equivariant
splitting $j$ of the sequence \eqref{eq:lieseq}, and a $G$-invariant
connection $\theta$ on $A$.  Replacing $\d$ with the equivariant
differential in the discussion above, we find that the 2-form $\varpi$
coincides with its equivariant extension. Its equivariant differential is
$\d_G\varpi=\a^*\eta_G$ where $\a^*\eta_G=-\l\d^\theta j,F^\theta_G\r$. Thus
\[ \a^*(\eta_G-\eta)=\l\d^\theta j,\Psi\r\]
where $\Psi\in\Om^0(N,L)$ was defined in \ref{subsec:equiv}.  To
address the equivariant lifting problem, we use the notation from the
proof of Theorem \ref{th:liftingproblem}.  Let $\wh{A}=\wh{L}\oplus
TN$ carry the diagonal $G$-action, the map $p\colon \wh{A}\to A$ is
$G$-equivariant. The bracket $[\cdot,\cdot]_{\hat A}$ defined by $\om$
with $\d\om=-\eta$, is $G$-invariant provided that $\om$ is
$G$-invariant.

\begin{theorem}
  Let $(A,\theta,x_A)$ be a $G$-equivariant transitive Lie algebroid,
  with invariant connection $\theta$ and with equivariant generators
  $x_A$. Let $\wh{L}\to L$ be a $G$-equivariant central extension,
  together with a $G$-equivariant splitting $j$, defining
  $\eta_G\in\Om^3_G(N)$ as above. Then there is a 1-1 correspondence
  between
\begin{enumerate}
\item isomorphism classes of equivariant lifts
      $(\hat{A},\hat{\theta},x_{\hat{A}})$ of the data
      $(A,\theta,x_A)$.
\item equivariant 2-forms $\om_G\in\Om^2_G(N)$ such that
      $\d_G\om_G=-\eta_G$.
\end{enumerate}
\end{theorem}
\begin{proof}
To describe generators for the action, it suffices to describe their
scalar part. Thus write
\[ x_{\wh{A}}=j(x_A)+\Phi(x)\]
for some $G$-equivariant map $\Phi\colon N\to \g^*$. Then $x_{\wh{A}}$ are generators for the
$\g$-action if and only if the map $x\mapsto x_{\hat A}$ 
defines the $\g$-representation on $\hat{A}$, i.e. 
\[ [x_{\hat A},\hat\xi]_{\hat A}=\L_{x_A}\hat\xi\]
for $\hat\xi\in\Gamma(\hat A)$. For $\hat\xi\in\Gamma(\wh{L})$, this
property is automatic. It is hence
enough to consider the condition 
\[ [x_{\hat A},\on{Hor}^{\hat \theta}(X)]_{\hat A}=\L_{x_A}\on{Hor}^{\hat\theta}(X)
=\on{Hor}^{\hat \theta}([x_N,X]).\]
Writing 
$ x_{\hat A}=j_A(x_A)+\Phi(x)=\on{Hor}^{\hat
  \theta}(x_N)+j(\Psi(x))+\Phi(x)$, 
we find, 
\[ \begin{split}
\lefteqn{[x_{\hat A},\on{Hor}^{\hat \theta}(X)]_{\hat A}-
\on{Hor}^{\hat \theta}([x_N,X])}\\
&=j(F^\theta(x_N,X))-\om(x_N,X)
-\L_X \Phi(x)+\L_{\on{Hor}^{\theta}(X)}j(\Psi(x))\\
&=-\om(x_N,X)-\iota_X \d\Phi(x)+\l\d^\theta j(X),\Psi(x)\r+\ldots
\end{split}
\]
where $\ldots$ indicates sections of $j_A(A)=0\oplus A$. The $\ldots$
terms have to cancel (by considering their image under $p$), hence we 
obtain the condition 
\[ \om(x_N,\cdot)+\d\Phi(x)=\l\d^\theta j(X),\Psi(x)\r.\] 
Since $\a^*\eta_G=\a^*\eta+\l\d^\theta j(X),\Psi\r$, this is the component of
form degree 1 of the equation $\d_G\om_G=-\eta_G$, where
$\om_G=\om-\Phi$ is an equivariant extension of $\om$.
\end{proof}

If $G$ is compact (so that invariant connections, splittings etc. can
be obtained by averaging), it follows that $[\eta_G]\in H^3_G(N,\R)$
is precisely the obstruction for the equivariant lifting problem.

Below, we will encounter situations where $\Phi=0$, so that $\om$
coincides with its equivariant extension. Equivalently, $j(x_A)$ 
are generators for the action. Here is a first example. 

\begin{example}
  Suppose $N=G/K$ for a compact subgroup $K$. Since
  $H_G(N,\R)=H_{K}(\pt,\R)$ vanishes in odd degrees, the class
  $[\eta_G]$ is necessarily trivial. Moreover, it is easy to see that
  $\eta_G=-\d_G \om$ for a unique invariant 2-form $\om$. Hence, one
  obtains a solution of the lifting problem with $x_{\hat A}=j(x_A)$.
\end{example}

\subsection{Relation with Courant algebroids}
In the previous sections, we explained how the lifting problem for a transitive Lie algebroid defines an obstruction class in $H^3(N,\R)$. By a well-known result of \v{S}evera, the group $H^3(N,\R)$ classifies exact Courant algebroids over $N$. We will explain now how to give a direct description of this Courant algebroid. 
As before, we start out by choosing a connection $\theta$ on $A$, as well as a splitting $j$. These data define a 2-form $\varpi\in \Gamma(\wedge^2 A^*)$, such that  $\d\varpi=\a^*\eta$ is basic. 

Let $A\oplus A^*$ carry the symmetric bilinear form extending the pairing between $A$ and $A^*$, and the standard Courant bracket, 
\[ \Cour{(v_1,\alpha_1),\ (v_2,\alpha_2)}=([v_1,v_2]_A,\ \L_{v_1}\alpha_2-\iota_{v_2}\d\alpha_1).\]
\begin{proposition}
  The map $f\colon L\to A\oplus A^*,\ \xi\mapsto
  (\xi,\iota_\xi\varpi)$ defines an isotropic $L$-action on $A\oplus
  A^*$. That is, its image is isotropic, and the induced map on
  sections preserves brackets.
\end{proposition}
\begin{proof}
Since $\d\varpi$ is basic, we have $\d\iota_\xi\varpi=L_\xi\varpi$. Hence
\[ \begin{split}
\Cour{f(\xi_1),f(\xi_2)}&=\Cour{(\xi_1,\iota_{\xi_1}\varpi),(\xi_2,\iota_{\xi_2}\varpi)}\\
&=([\xi_1,\xi_2]_A,\ L_{\xi_1}\iota_{\xi_2}\varpi-\iota_{\xi_2}\d\iota_{\xi_1}\varpi) \\
&=([\xi_1,\xi_2]_A,\ \iota_{[\xi_1,\xi_2]_A}\varpi)=f([\xi_1,\xi_2]_A).
\end{split}\]
The property $\l f(\xi),f(\xi)\r=0$ is straightforward.
\end{proof}

As in Bursztyn-Cavalcanti-Gualtieri \cite{bur:red} we may consider the
reduction of $A$ by the isotropic $L$-action. 
\begin{proposition}
The reduced Courant algebroid $f(L)^\perp/f(L)$ is canonically
isomorphic to $TN\oplus T^*N$ with the $\eta$-twisted Courant
bracket. 
\end{proposition}
\begin{proof}
  Let $f\colon A\to A\oplus A^*,\ v\mapsto (v,\iota_v\varpi)$ be the
  obvious extension of the action map. Then $f(L)^\perp=f(A)+T^*N$,
  where $T^*N$ is embedded as the annihilator of $L$ in $A^*$, and
  hence $ f(L)^\perp/f(L)=f(A)/f(L)\oplus T^*N=TN\oplus T^*N$.  For
  $v_1,v_2\in \Gamma(A)$ and if $\alpha_1,\alpha_2\in \Gamma(T^*N)\cong\Gamma(A^*)_{\on{basic}}$ we have
\[ \begin{split}
\Cour{f(v_1)+\alpha_1,f(v_2)+\alpha_2}&=
([v_1,v_2]_A,\L_{v_1}\iota_{v_2}\varpi-\iota_{v_2}\d\iota_{v_1}\varpi+L_{v_1}\alpha_2-\iota_{v_2}\d 
\alpha_1)\\
&=([v_1,v_2]_A,\ \iota_{[v_1,v_2]_A}\varpi+\iota_{v_2}\iota_{v_1}\a^*\eta+L_{v_1}\alpha_2-\iota_{v_2}
\d \alpha_1)\\
&=f([v_1,v_2]_A)+\iota_{v_2}\iota_{v_1}\a^*\eta+L_{v_1}\alpha_2-\iota_{v_2}\d \alpha_1.
\end{split}\]
This shows that the Courant bracket on $\Gamma(f(L)^\perp/f(L))$ is
the $\eta$-twisted Courant bracket on $TN\oplus T^*N$.
\end{proof}

\section{The Atiyah algebroid $A\to G$}
\subsection{The bundle of twisted loop algebras}
Let $G$ be a Lie group. For $g\in G$ define the twisted loop algebra
\[ L_g=\{\xi\in C^\infty(\R,\g)|\ \xi(t+1)=\Ad_g\xi(t)\},\]
with bracket $[\xi_1,\xi_2]_L(t)=-[\xi_1(t),\xi_2(t)]_\g$ \emph{minus}
\footnote{The sign change will be convenient for what follows. It is
  related to the appearance of the minus sign in Example
  \ref{ex:princ}.} the pointwise Lie bracket on $C^\infty(\R,\g)$.
Let $L\to G$ be the Lie algebra bundle with fibers $L_g$.  (The
isomorphism type of the fiber $L_g$ may depend on the connected
component of $G$ containing $g$.)

\begin{remark}
  Let us discuss briefly the local triviality of $L$.  Consider a
  connected component of $G$, with base point $g_0$.  For any $g$ in
  the same connected component, the choice of any path
  $\gamma=\gamma_{g}\in C^\infty([0,1],G)$ from $\gamma(0)=g_0$ to
  $\gamma(1)=g$, with $\gamma$ constant near $t=0,1$, defines a Lie
  algebra isomorphism
\[ L_{g_0}\to L_g,\ \xi\mapsto \ti{\xi}\]
where $\ti{\xi}(t)=\Ad_{\gamma(t)}\xi(t)$ for $t\in [0,1]$.  One may
take $\gamma_{g}(t)$ to depend smoothly on $g,t$ (as $g$ varies in a
small open subset), thus obtaining local trivializations of $L$.  The
smooth sections of $L$ are thus functions $\xi\in
C^\infty(G\times\R,\g)$ satisfying $\xi(g,t+1)=\Ad_g \xi(g,t)$.
\end{remark}

\subsection{The Lie algebroid $A\to G$}
Let $\theta^L,\theta^R\in \Om^1(G,\g)$ be the Maurer-Cartan forms on
$G$. We will work with the right trivialization of the tangent bundle,
$TG\to G\times\g,\ X\mapsto \iota_X\theta^R$. Note that if $X,Y\in\mf{X}(G)$
correspond to $v=\iota_X \theta^R,\ w=\iota_Y\theta^R$, then $[X,Y]$
corresponds to 
\[ \iota_{[X,Y]}\theta^R=-[v,w]_\g+X v-Yw\]
where the subscript $\g$ indicates the pointwise bracket. For $g\in G$ let
\[ A_g=\{\xi\in C^\infty(\R,\g)|\ \exists v_\xi\in \g\colon 
\xi(t+1)=\Ad_g\xi(t)+v_\xi\}.\]
We obtain an exact sequence
\[ 0\to L_g\to A_g \xrightarrow{\a} T_gG\to 0\]
where the anchor map $\a$ is defined by $\iota_{\a(\xi)}\theta^R=v_\xi$. 
Let $A\to G$ be the bundle with fibers $A_g$. 

\begin{proposition}
The bundle $A$ with anchor map $\a$ is a
transitive Lie algebroid over $G$, with bracket on 
\[ \Gamma(A)=\{\xi\in C^\infty(G\times \R,\g)|\ \exists v_\xi\in C^\infty(G,\g)\colon 
\xi(t+1)=\Ad_g\xi(t)+v_\xi\}.\]
given by
\[ [\xi,\zeta]_A=-[\xi,\zeta]_\g+X\zeta-Y\xi.\]
Here $X,Y\in\mf{X}(G)$ are determined by $v_\xi=\iota_X\theta^R,\ 
v_\zeta=\iota_Y\theta^R$. 
\end{proposition}
\begin{proof}
To show that $\a$ is surjective, fix $f\in
C^\infty([0,1],\R)$ with $f(0)=0,f(1)=1$, and constant near $t=0,1$.
For $X\in T_gG$ let $\xi(t)=f(t)\iota_X\theta^R$ for $t\in
[0,1]$. Then $\xi(1)-\xi(0)=\iota_X\theta^R$, so $\xi$ extends
uniquely to an element $\xi\in A_g$ with $\a(\xi)=X$. This argument
also verifies that $A$ is locally trivial, in fact $A\cong L\oplus TG$.

To check that $[\cdot,\cdot]_A$ preserves the space
$\Gamma(A)\subset C^\infty(G\times\R,\g)$, we calculate
(at any given $g\in G$)
\[ [\xi(t+1),\zeta(t+1)]_{\g}
=\Ad_g [\xi(t),\zeta(t)]_{\g}+[\Ad_g\xi(t),v_\zeta]_{\g}+[v_\xi,\Ad_g\zeta(t)]_{\g}+[v_\xi,v_\zeta]_{\g}.\]
On the other hand, 
\[ \begin{split}(X\zeta)(g,t+1)&=\f{\p}{\p
    s}|_{s=0}\Big(\zeta\big(\exp(sv_\xi(g))g,t+1)\big)\Big)\\
  &=\f{\p}{\p
    s}|_{s=0}\Big(\Ad_{\exp(sv_\xi(g))g}\big(\zeta(\exp(sv_\xi(g))g,t)\big)+v_\zeta(\exp(sv_\xi(g))g)\Big)\\
  &=[v_\xi,\Ad_g\zeta(t)]_{\g} +\Ad_g(X\zeta)(t)+X v_\zeta, 
\end{split}\]
with a similar expression for $(Y\xi)(g,t+1)$. This verifies
\[ [\xi,\zeta]_A(t+1)=\Ad_g([\xi,\zeta]_A(t))+v_{[\xi,\zeta]_{\mf{X}}}.\]
This shows that $[\cdot,\cdot]_A$ takes values in $A$ and also that
$\a([\xi,\zeta]_A)=[\a(\xi),\a(\zeta)]$. It is straightforward to check
that $[\cdot,\cdot]_A$ obeys the Jacobi identity. 
\end{proof}

\begin{proposition}
The Lie group $G$ acts on $A\to G$ by Lie algebroid automorphisms 
covering the conjugation action on $G$. This action is given on 
sections $\xi\in\Gamma(A)$ by 
\[(k.\xi)(g,t)=\Ad_k\xi(\Ad_{k^{-1}}g,t),\ \xi\in\Gamma(A),\,k\in G,\]
and has infinitesimal generators 
\[ \g\to \Gamma(A),\ \ x\mapsto x_A=-x.\]
\end{proposition}
\begin{proof}
Let $X\in \mf{X}(G)$ be the vector field such that
$\iota_X\theta^R=v_\xi$, and let $k.X:=(\d\Ad_k)(X)$ its push-forward
under the conjugation action. Then 
\[ (\iota_{k.X}\theta^R)(g)=\Ad_k\big((\iota_X\theta^R)(\Ad_{k^{-1}}g)\big).\]
The following calculation shows that the action is well-defined, and that the anchor map is equivariant: 
\[ 
\begin{split}
  (k.\xi)(g,t+1)&=\Ad_k\big(\xi(\Ad_{k^{-1}}g,t+1)\big)\\
  &=\Ad_k\big(\Ad_{k^{-1}gk}\xi(\Ad_{k^{-1}}g,t)+(\iota_X\theta^R)(\Ad_{k^{-1}}g)\big)\\
  &=\Ad_g\big((k.\xi)(g,t)\big)+(\iota_{k.X}\theta^R)(g)
\end{split}
\]
It is straightforward to check that
$k.[\xi,\zeta]_A=[k,\xi,k.\zeta]_A$.  For $x\in\g$ the generating
vector field $x_G=x^L-x^R$ for the conjugation action on $G$ satisfies
$\iota_{x_G}\theta^R=\Ad_g(x)-x$. Hence $x_A(g,t):=-x$
defines a section of $A$, with $\a({x_A})=x_G$. For all $\xi\in\Gamma(A)$, 
\[ [{x_A},\xi]_A(g,t)=[x,\xi(g,t)]_\g+x_G \xi 
=\f{\p}{\p u}\Big|_{u=0} (\exp(ux).\xi)(g,t), 
\]
confirming that the map $x\mapsto {x_A}$ gives generators for the action. 
\end{proof}

\subsection{The bundle $PG\to G$}
We will now interpret $A\to G$ as the Atiyah algebroid of a principal
bundle over $G$. Suppose first that $G$ is connected. Let $\pi\colon
PG\to G$ be the bundle with fibers,
\[ (PG)_g=\{\gamma\in C^\infty(\R,G)|\ 
\gamma(t+1)=g\gamma(t)\}.\]
By an argument similar to that for $L_g$, one sees that $PG\to G$ is a
locally trivial bundle. In fact it is a principal bundle with fiber
the loop group $LG=\pi^{-1}(e)$. We will argue that $A\to G$ may be
regarded as the Atiyah algebroid of the principal $LG$-bundle $PG\to
G$.  Let $\gamma\in (PG)_g$.  Given a family of paths $\gamma_s\in PG$
with $\gamma_0=\gamma$, let $\zeta\colon \R\to \g$ be defined as
\[ \zeta(t)=\f{\p}{\p s}|_{s=0} \big(\gamma_s(t)\gamma(t)^{-1}\big).\]
Put $g_s=\pi(\gamma_s)$, so that $g_s=\gamma_s(t+1)\gamma_s(t)^{-1}$
for all $t$. We may write $\gamma_s(t)=\exp(s\zeta_s(t))\gamma(t)$, so that
$\zeta_0(t)=\zeta(t)$. Then 
\[\begin{split}
 g_s&=\exp(s\zeta_s(t+1))\gamma(t+1)\gamma(t)^{-1}\exp(-s\zeta_s(t))\\&=
\exp(s\zeta_s(t+1))g\exp(-s\zeta_s(t)).\end{split}\]
We find, 
\[ \f{\p}{\p s}|_{s=0} \big(g_s g^{-1}\big)
=\zeta(t+1)-\Ad_g \zeta(t).\]
This identifies $A_g$ as the space of maps for which $\zeta(t+1)-\Ad_g
\zeta(t)$ is constant. The formula for $[\cdot,\cdot]_A$ is the
expected bracket on $LG$-invariant vector fields on $PG$. However,
rather than attempting to construct Lie brackets of vector fields on
infinite-dimensional manifolds, we will take this formula simply as a
definition.

\begin{remark}
  If $G$ is disconnected, the condition $\gamma(t+1)=g\gamma(t)$
  implies that $g$ is in the identity component. One may however
  extend the definition, as follows: For any given component of $G$,
  pick a base point $g_0$, and take $(PG)_g$ (with $g$ in the
  component of $g_0$ to consist of paths $\gamma$ such that
  $\gamma(t+1)=g\gamma(t)g_0^{-1}$. Then $PG\to G$ is a principal
  $L_{g_0}$-bundle over the given component.
\end{remark}

\subsection{Connections on $A\to G$}
Let us next discuss connections on the Atiyah algebroid over $ G$.  It
will be convenient to describe $\theta$ in terms of the horizontal
lift, $\on{Hor}^\theta\colon TG\to A\subset C^\infty(\R,\g)$. Write
$\on{Hor}^\theta=-\alpha$, and think of $\alpha$ as a family of
1-forms $\alpha_t\in\Om^1(G,\g)$.
\begin{lemma}
A family of 1-forms $\alpha_t$ defines a horizontal lift $TG\to A$ if
and only if 
\begin{equation}\label{eq:alpha}
 \alpha_{t+1}=\Ad_g\alpha_t-\theta^R=:g\bullet \alpha_t.
\end{equation}
Here $\bullet$
denotes the `gauge action' of the identity map $g\in C^\infty(G,G)$ 
on $\Om^1(G,\g)$. The resulting connection is $G$-equivariant if and
only if $\alpha_t\in\Om^1(G,\g)^G$. 
\end{lemma}
\begin{proof}
 The condition for $\on{Hor}^\theta=-\alpha$ to define a horizontal lift is that
for all $X\in\mf{X}(G)$, 
\[ -\iota_X\alpha_{t+1}=-\Ad_g(\iota_X\alpha_t)+\iota_X\theta^R
=-\iota_X(\Ad_g(\alpha_t)-\theta^R|_g)
\]
for all such $X$. This gives the condition on $\alpha_t$. It is clear
that $\on{Hor}^\theta$ is $G$-equivariant exactly if $\alpha$ is
$G$-equivariant.
\end{proof}

The connection $\theta=\theta^\alpha\colon A\to L$ defined by $\alpha$ is 
$\theta(\xi)=\xi+\alpha(\a(\xi))$. 
Let $F^{\alpha_t}=\d\alpha_t+\hh [\alpha_t,\alpha_t]_\g$ be the curvature
of $\alpha_t$. By the property of the curvature under gauge
transformations, 
\[ F^{\alpha_{t+1}}=F^{g\bullet\alpha_t}=\Ad_g F^{\alpha_t}.\]
\begin{proposition}\label{prop:curv}
The curvature $F^\theta\in \Om^2(G,L)$ of the connection
$\theta(\xi)=\xi+\alpha(\a(\xi))$ is given by 
\[ F^\theta(X,Y)(t)=F^{\alpha_t}(X,Y),\ \ X,Y\in\mf{X}(G).\]
If $\alpha$ is $G$-invariant, then the corresponding map
$\Psi\colon\g\to \Gamma(L)$ (cf. \ref{subsec:equivariant}) is
\[ \Psi(x)=-x+\iota(x_N)\alpha.\]
\end{proposition}
\begin{proof}
This follows from the definition of the curvature in terms of
horizontal lifts: 
\[\begin{split} F^\theta(X,Y)&=\on{Hor}_\theta([X,Y])-[\on{Hor}_\theta(X),\on{Hor}^\theta(Y)]_A\\
&=-\alpha([X,Y])-[\alpha(X),\alpha(Y)]_A\\
&=-\alpha([X,Y])+[\alpha(X),\alpha(Y)]_\g+X\alpha(Y)-Y\alpha(X)\\
&=(\d\alpha+\hh[\alpha,\alpha])(X,Y).\end{split}\]
If $\alpha$ is $G$-invariant, so that $\theta$ is $G$-equivariant, the
map $\Psi(x)=-\iota_{x_A}\theta$ is given as
\[ \Psi(x)=-\iota_{x_A}\theta=x_A-\on{Hor}^\theta(x_N)=-x+\iota(x_N)\alpha.\]
\end{proof}
To construct a family of 1-forms $\alpha_t\in\Om^1(G,\g)$ with the
transformation property \eqref{eq:alpha}, take any $\alpha_0$ (for
example $\alpha_0=0$), and put
$\alpha_n=g^n\bullet \alpha_0$. Pick a smooth function $f\colon [0,1]\to \R$ such
that $f(t)=0$ near $t=0$ and $f(t)=1$ near $t=1$, and let 
\begin{equation}\label{eq:standardalpha}
\alpha_t =\alpha_n+f(t-n)(\alpha_{n+1}-\alpha_n)\end{equation}
for $n\le t\le n+1$. The resulting $\alpha_t$ is smooth, and has the
desired transformation property. If $\alpha_0\in\Om^1(G,\g)$ is $G$-invariant, then
\[ k^*\alpha_n=(\Ad_k(g))^n\bullet k^*\alpha_0
=(\Ad_k(g^n))\bullet \Ad_k(\alpha_0)=\Ad_k(\alpha_n),\]
hence $\alpha_t\in\Om^1(G,\g)$ is $G$-invariant for all $t$.

\section{The lifting problem for $A\to G$}
An invariant inner product on $\g$ defines central extensions
$\wh{L}_g$ of the twisted loop algebras $L_g$. In this Section, we
will work out the 2-form $\varpi\in \Gamma(\wedge^2 A^*)$ defined by
the lifting problem, and discuss some of its properties.

\subsection{Central extensions}
Suppose the Lie algebra $\g$ carries an invariant symmetric bilinear
form $\cdot$ (possibly indefinite, or even degenerate). This then
defines a central extension
\begin{equation}\label{eq:L} 0\to G\times \R\to \wh{L}\to L\to 0,\end{equation}
where $\wh{L}_g=L_g\oplus \R$ with bracket, 
\[ [(\xi_1,s_1),(\xi_2,s_2)]_{\wh{L}}=\Big(-[\xi_1,\xi_2]_\g,\
\int_0^1\dot{\xi}_1\cdot\xi_2\Big).\]
Here $\dot{\xi}=\f{\p \xi}{\p t}$, and the integral is relative to the
measure $\d t$. The $G$-action on $L$ lifts to an action on $\wh{L}$,
by $k.(\xi,s)=(k.\xi,s)$. 
\begin{proposition}
  The representation of $A$ on $L$ (given by $\nabla_\xi\zeta=[\xi,\zeta]_A$)
lifts to the Lie algebra bundle
  $\wh{L}$, by the formula,
\[ \wh{\nabla}_\xi(\zeta,s)=\Big(\nabla_\xi\zeta,\ \a(\xi)s+\int_0^1 \dot\xi\cdot\zeta  \Big)\]
for $\xi\in\Gamma(A),\ (\zeta,s)\in \Gamma(\wh{L})$. This
representation is equivariant relative to the $G$-actions on $A,\
\wh{L}$. 
\end{proposition}
\begin{proof}
  We first verify that this formula defines an $A$-representation.
  Clearly $\xi\mapsto \wh{\nabla}_{\xi}$ is $C^\infty(N)$-linear. For
  $\xi_1,\xi_2\in \Gamma(A)$ and $\zeta\in \Gamma(L)$, we have
\[ \int_0^1 [\xi_1,\dot{\xi}_2]_A\cdot\zeta=-\int_0^1 \dot{\xi}_2\cdot[\xi_1,\zeta]_A
+\a(\xi_1)\int_0^1 \dot{\xi}_2\cdot \zeta,\]
by the definition of the bracket on $A$ and the $\Ad$-invariance
of $\cdot$. Note that $\a(\dot{\xi})=0$ for all $\xi \in \Gamma(A)$.
Subtracting a similar equation with $1\leftrightarrow 2$ interchanged, 
one obtains 
\[ \int_0^1 \f{\p}{\p t}([\xi_1,\xi_2]_A) \cdot \zeta=
\a(\xi_1)\int_0^1 \dot{\xi}_2\cdot\zeta
-  \a(\xi_2)\int_0^1 \dot{\xi}_1\cdot\zeta+
\int_0^1 (\dot{\xi}_1\cdot\nabla_{\xi_2}\zeta
-\dot{\xi}_2\cdot
\nabla_{\xi_1}\zeta)
\]
which easily implies the property
$\wh\nabla_{\xi_1}\wh\nabla_{\xi_2}-\wh\nabla_{\xi_2}\wh\nabla_{\xi_1}=\wh\nabla_{[\xi_1,\xi_2]_A}$.
We next check that this representation acts by derivations of the Lie
bracket on $\Gamma(\wh{L})$. We have
\[\begin{split}
  [\wh\nabla_\xi(\zeta_1,s_1),(\zeta_2,s_2)]_{\wh{L}}
  &=\Big(-[\nabla_\xi\zeta_1,\zeta_2]_\g,\ \int_0^1 \big(
  -\f{\p}{\p t}[\xi,\zeta_1]_\g+\a(\xi)\dot{\zeta}_1\big)\cdot \zeta_2 \Big),\\
  [(\zeta_1,s_1),\wh\nabla_\xi(\zeta_2,s_2)]_{\wh{L}}&=\Big(-[\zeta_1,\nabla_\xi\zeta_2]_\g,\ 
  \int_0^1 \dot{\zeta}_1\cdot \big(-[\xi,\zeta_2]_\g
  +\a(\xi)\zeta_2\big)\Big),
\end{split}\]
which adds up to
\[ \wh\nabla_\xi[(\zeta_1,s_1),(\zeta_2,s_2)]_{\wh{L}}=\Big(-\nabla_\xi[\zeta_1,\zeta_2]_\g,\ \a(\xi)\int_0^1
\dot{\zeta}_1\cdot\zeta_2   -\int_0^1
\dot{\xi}\cdot[\zeta_1,\zeta_2]_\g \Big)
\]
as required. Equivariance of the action is clear.
\end{proof}

%\begin{proof}
%We compute, 
%\[ 
%\begin{split}
%\xi_1.(\xi_2.(\zeta,s))
%&=\xi_1.\Big([\xi_2,\zeta]_A,\ a(\xi_2)s-\int_0^1 \dot{\xi}_2\cdot\zeta  \Big)\\
%&=\Big([\xi_1,[\xi_2,\zeta]_A]_A,\ a(\xi_1)a(\xi_2)s
%-a(\xi_1)\int_0^1 \dot{\xi}_2\cdot\zeta  
%-\int_0^1\dot{\xi}_1\cdot [\xi_2,\zeta]_A\Big).
%\end{split}\]
%\end{proof}

By definition, $\wh{L}$ comes with the $G$-equivariant splitting
$j\colon L\to \wh{L},\ \xi\mapsto (\xi,0)$, with associated cocycle
\[ \sig(\xi_1,\xi_2)=-\int_0^1 \dot{\xi}_1\cdot \xi_2.\  
\] 
Let $\alpha_t\in\Om^1(G,\g)$ be a family of 1-forms with the
transformation property \eqref{eq:alpha} and let $\theta^\alpha\colon
A\to L$ the corresponding connection.  Using the results from the last
Section, we obtain a 2-form $\varpi^\alpha\in \Gamma(\wedge^2 A^*)$
and a closed 3-form $\eta^\alpha\in \Om^3(G)$, whose cohomology class
is the obstruction to the existence of a lift $\wh{A}$.  If $\alpha$
is $G$-equivariant, we also obtain an equivariant extension
$\eta_G^\alpha$ of the 3-form.  We will now derive explicit formulas.

\subsection{The 2-form $\varpi^\alpha$}\label{subsec:kappat}
To begin, we need the covariant derivative $\d^{\theta^\alpha}
j\in\Om^1(G,L^*)$ of the splitting. Note that the derivative
$\dot{\alpha}_t$ satisfies $\dot{\alpha}_{t+1}=\Ad_g\dot{\alpha}_t$,
so it defines an element $\dot{\alpha}\in \Om^1(G,L)$.
\begin{lemma}\label{lem:zeta}
For $\zeta\in \Gamma(L)$ one has
\[ \l\d^{\theta^\alpha} j,\zeta\r=-\int_0^1 \dot{\alpha}\cdot\zeta\  .\]
\end{lemma}
\begin{proof}
  Recall that $\l\d^{\theta^\alpha} j,\zeta\r=\l \d
  j,\zeta\r+\sig({\theta^\alpha},\zeta)$.  For $\xi\in \Gamma(A)$ we
  compute
\[ \begin{split}
  \iota_\xi\l \d j,\zeta\r&=\L_\xi j(\zeta)-j(\L_\xi\zeta)
  =\int_0^1 \dot{\xi}\cdot \zeta,\\
  \sig(\iota_\xi{\theta^\alpha},\zeta)&=
  \sig(\xi+\iota_{\a(\xi)}\alpha,\zeta) =-\int_0^1
  \iota_{\a(\xi)}\dot{\alpha}\cdot \zeta-\int_0^1 \dot{\xi}\cdot
  \zeta.\qedhere\end{split} \]
\end{proof}

Equation \eqref{eq:brylinski}  together with this Lemma shows that 
\[ \varpi^\alpha=-\int_0^1 \dot{\alpha}\cdot \theta^\alpha-\hh\sig(\theta^\alpha,\theta^\alpha).\]
It is convenient to introduce the forms $\kappa_t\in
(\Gamma(A^*)\otimes\g)^G$, 
\[ \kappa_t(\xi)=-\xi_t,\ \ \xi\in \Gamma(A).\]
\begin{lemma}\label{lem:kappa}
The forms $\kappa_t$ satisfy $F_G^{\kappa_t}(x)+x=0$, and 
\[\kappa_{t+1}=\Ad_g(\kappa_t)-\a^*\theta^R=:
g\bullet\kappa_t.\]
\end{lemma}
\begin{proof}
We have
\[
\d\kappa_t(\xi,\zeta)=-\kappa_t([\xi,\zeta]_A)+\a(\xi)\kappa_t(\zeta)-\a(\zeta)\kappa_t(\xi)=-[\xi,\zeta]_\g
=-\hh[\kappa_t,\kappa_t](\xi,\zeta).\]
This shows $F^{\kappa_t}=0$. Furthermore, for $x\in \g$ we have
$\iota(x_A)\kappa_t=-x$,
by definition of $x_A$. Hence $F_G^{\kappa_t}(x)+x=0$.
The transformation property $\kappa_{t+1}=\Ad_g\kappa_t-\a^*\theta^R$
follows from the definition of $A$. 
\end{proof}

Let $Q^\alpha\in\Om^2(G)$ be the 2-form (see Section \ref{subsec:CS1})
\begin{equation}\label{eq:qalpha}
 Q^\alpha=\hh \theta^L\cdot\alpha_0+\hh \int_0^1 \alpha_t
 \cdot\dot{\alpha}_t,\end{equation}
and define $Q^\kappa\in \Gamma(\wedge^2 A^*)$ by a similar expression,
with $\alpha_t$ replaced by $\kappa_t$.

\begin{proposition}
We have $\varpi^\alpha=\a^* Q^\alpha-Q^\kappa$.
\end{proposition}
\begin{proof}
By definition, $\theta^\alpha=\a^*\alpha-\kappa$.
To simplify notation, we omit the pull-back $\a^*$ in the following
computation, i.e. we view $\Om(G)$ as a subspace of $\Gamma(\wedge A^*)$:
\[ \begin{split}\varpi^\alpha&=-\int_0^1 \dot{\alpha}\cdot(\alpha-\kappa)+\hh\int_0^1
(\dot{\alpha}-\dot{\kappa})\cdot (\alpha-\kappa)\\
&=\hh \int_0^1 \alpha\cdot \dot{\alpha}-\hh \int_0^1
\kappa \cdot  \dot{\kappa}
-\hh \int_0^1 \f{\p}{\p t}(\kappa\cdot\alpha)\\
&=\hh \int_0^1 \alpha\cdot \dot{\alpha}-\hh \int_0^1
\kappa \cdot  \dot{\kappa}
-\hh(\Ad_g\kappa_0-\theta^R)\cdot(\Ad_g\alpha_0-\theta^R)
+\hh \kappa_0\cdot\alpha_0\\
&=Q^\alpha-Q^\kappa.
\end{split}
.\]
\end{proof}

\begin{lemma}
For $\alpha$ as in \eqref{eq:standardalpha}, one has
\[ Q^\alpha=\Big(\f{\theta^L+\theta^R}{2}\Big)\cdot \alpha_0+\hh
\alpha_0\cdot\Ad_g\alpha_0.\]
In particular, $Q^\alpha=0$ for $\alpha_0=0$. 
\end{lemma}
\begin{proof}
By assumption, 
$\alpha_t=\alpha_0+f(t)(g\bullet\alpha_0-\alpha_0)$ for $0\le t\le 1$,
where $f(0)=0$ and $f(1)=1$. Hence 
$\alpha_t\cdot\dot{\alpha}_t
=\dot{f}\alpha_0\cdot(g\bullet\alpha_0)$,  and therefore 
\[\hh \int_0^1 \alpha_t\cdot\dot{\alpha}_t=\hh
\alpha_0\cdot(g\bullet\alpha_0)
=\hh \alpha_0\cdot\Ad_g\alpha_0+\hh\theta^R\cdot\alpha_0
.\]
Adding $\hh \theta^L\cdot\alpha_0$, the formula for $Q^\alpha$ follows.
\end{proof}

For the rest of this
paper, we will write $\varpi:=-Q^\kappa\in \Gamma(\wedge^2 A^*)$,
that is
\begin{equation}\label{eq:varpiformula}
 \varpi=-\hh \int_0^1 \kappa\cdot \dot{\kappa}-\hh 
\a^*\theta^L\cdot 
\kappa_0.\end{equation}
Thus $\varpi^\alpha=\varpi$ for any choice of $\alpha$ with
$Q^\alpha=0$. 
More explicitly, for $\xi\in \Gamma(A)$ we have
\[ \begin{split}
\varpi(\xi,\cdot)&=\hh \int_0^1 (\xi\cdot
\dot{\kappa}-\dot{\xi}\cdot \kappa)
-\hh v_\xi\cdot \Ad_g\kappa_0-\hh \Ad_g\xi_0\cdot
\mathsf{a}^*\theta^R\\
&=-\int_0^1 \dot{\xi}\cdot\kappa+\hh(\xi_1\cdot\kappa_1-\xi_0\cdot\kappa_0)
-\hh v_\xi\cdot \Ad_g\kappa_0-\hh \Ad_g\xi_0\cdot
\mathsf{a}^*\theta^R\\
%&=-\int_0^1
%\dot{\xi}\cdot\kappa+\hh\big((\Ad_g\xi_0+v_\xi)\cdot(\Ad_g\kappa_0-\mathsf{a}^*\theta^R)\big)
%-\hh \xi_0\cdot\kappa_0
%-\hh v_\xi\cdot \Ad_g\kappa_0-\hh \Ad_g\xi_0\cdot
%\mathsf{a}^*\theta^R\\
&=-\int_0^1\dot{\xi}\cdot\kappa-\Ad_g(\xi_0)\cdot\mathsf{a}^*\theta^R-\hh
v_\xi\cdot \mathsf{a}^*\theta^R
\end{split}\]
Taking another contraction with $\zeta\in\Gamma(A)$, 
\[ \varpi(\xi,\zeta)=\int_0^1 \dot{\xi}\cdot\zeta-\Ad_g(\xi_0)\cdot v_\zeta-\hh
v_\xi\cdot v_\zeta.\]

\subsection{The 3-form $\eta^\alpha$}
Let $\eta\in\Om^3(G)$ be the \emph{Cartan 3-form} on $G$ given as 
\[ \eta=\f{1}{12}\theta^L\cdot[\theta^L,\theta^L]\in\Om^3(G),\]
and let $\eta_G\in\Om^3_G(G)$ be its equivariant extension 
\[ \eta_G(x)=\eta-\hh (\theta^L+\theta^R)\cdot x.\]
The 2-form $\varpi=-Q^\kappa\in\Gamma(\wedge^2 A^*)$ obeys
\[ \d_G \varpi(x)=-\d_G Q^\kappa(x)=\a^*\eta_G(x)-\int_0^1 \dot\kappa_t \cdot
(F^{\kappa_t}_G(x)+x)=\a^*\eta_G(x),\]
in particular $\d\varpi=\a^*\eta$. We obtain:
\begin{theorem}\label{th:3form}
  We have $\d\varpi^\alpha=\a^*(\eta+\d Q^\alpha)$, and if $\alpha$ is
  $G$-invariant, $\d_G\varpi^\alpha=\a^*(\eta_G+\d_G Q^\alpha)$. In
  particular, taking an invariant $\alpha$ with $Q^\alpha=0$, the
  2-form $\varpi\in\Gamma(\wedge^2 A^*)^G$ defined in
  \eqref{eq:varpiformula} satisfies
\[ \d_G\varpi=\a^*\eta_G.\]
\end{theorem}

\section{Fusion}\label{subsec:mult}
In this Section, we will study multiplicative properties of the Atiyah
algebroid over $G$, and of the forms $\varpi$. We begin by introducing
a (partial) multiplication on $A$, using concatenation of paths.
Let $\xi'\in
A_{g'},\ \xi''\in A_{g''}$, with
\[ \xi'_1=\xi''_0.\]
The \emph{concatenation} $\xi''*\xi'$ is defined as follows: 

\[ (\xi''*\xi')_t=\begin{cases} \xi'_{2t}&\text{ if $0\le t\le
    \hh$}\\
\xi''_{2t-1}&\text{ if $\hh\le t\le
    1$}\\
\end{cases}\]
extended to all $t$ by the property, 
\[ (\xi''*\xi')_{t+1}
=\Ad_{g'' g'} (\xi''*\xi')_t+(\Ad_{g''}v_{\xi'}+v_{\xi''}).\]
This is consistent since, putting $t=0$,   
\[ \xi''_1=\Ad_{g''}\xi''_0+v_{\xi''}=
\Ad_{g'' g'}\xi'_0+(\Ad_{g''}v_{\xi'}+v_{\xi''}).\]
Then $\xi''*\xi'\in A_{g'' g'}$ provided the concatenation is
\emph{smooth}.  The concatenation is smooth if, for example,
$\xi'',\xi'$ are constant near $t=0$.
Let 
\[ A^{[2]}\subset A\times A\] 
be the sub-bundle of \emph{composable paths}, with fiber at $(g'',g')$ the set of pairs
$(\xi'',\xi')\in A_{g''}\times A_{g'}$ such that $\xi'_1=\xi''_0$ and
such that $\xi''*\xi'$ is smooth. 
One easily checks that $A^{[2]}$ is a Lie subalgebroid of $A\times A$,
i.e. that the bracket on $\Gamma(A\times A)$ restricts to
$\Gamma(A^{[2]})$. The kernel of its anchor map $\a^{[2]}\colon
A^{[2]}\to TG^2$ is denoted $L^{[2]}$; it is a sub Lie algebra bundle
of $L\times L$.

Concatenation gives a bundle map $\on{mult}_A\colon A^{[2]}\to A$, 
covering the group multiplication $\on{mult}_G\colon G\times G\to G$. 
That is, we have a commutative diagram, 
\[ \begin{CD} 
A^{[2]} @>>{\on{mult}_A}> A\\
@VVV @VVV\\
G\times G @>>{\on{mult}_G}> G
\end{CD}\]
We have three transitive Lie algebroids over $G^2$, with inclusion maps
\begin{equation}\label{eq:3bun} 
A^2 \leftarrow A^{[2]} \rightarrow \on{mult}_G^!A,
\end{equation}
Here the left map is given by the definition of $A^{[2]}$, while the
right map is concatenation. The two maps correspond to reductions of
the structure Lie algebroids to $L^{[2]}$, 
\footnote{By analogy, one
  may think of $L^{[2]}$ as `figure eight' loops. The two maps
  correspond to viewing the figure eight either as a single loop or as
  a pair of two loops.}
\[ L^2\leftarrow L^{[2]}\rightarrow \on{mult}_G^* L.\] 
We are interested in compatible principal connections on the 
three transitive Lie algebroids \eqref{eq:3bun} over $G\times G$.
Write the elements of $G^2$ as $(g'',g')$, and use the similar
notation to indicate projections to the two factors.  Let
$\alpha',\alpha''\colon \R \to \Om^1(G^2,\g)^G$ be smooth families of
1-forms with
\[ \alpha'_{t+1}=g'\bullet\alpha'_t,\ \ \alpha''_{t+1}=g''\bullet \alpha''_t.\] 
Assume both of these are constant near $t=0$ (hence near any integer
$t=n$), and that $\alpha'_1=\alpha''_0$. The concatentation (cf. Prop.
\ref{prop:q}) $\alpha''*\alpha'\colon \R\to \Om^1(G^2,\g)^G$ defines a
connection $\theta^{\alpha''*\alpha'}$ on $m_G^!A$, while the pair
$\alpha'',\alpha'$ defines a connection $\theta^{\alpha'',\alpha'}$ on
$A\times A$. These two connections are compatible, in the sense that
they restrict to the same connection on $A^{[2]}$.
For the corresponding forms $\varpi^{\alpha'}$ etc. this implies
\[ \varpi^{\alpha''*\alpha'}\Big|_{A^{[2]}}=
\varpi^{\alpha'',\alpha'}\Big|_{A^{[2]}}     ,\]
and hence the 3-forms satisfy 
$ \on{mult}_G^*\eta^{\alpha''*\alpha'}=
\eta^{\alpha'',\alpha'}$.

Let $\varpi\in \Gamma(\wedge^2 A^*)$ be the 2-form defined 
in \eqref{eq:varpiformula}. Then 
\[ \begin{split}
\varpi^{\alpha''*\alpha'}&=
\on{mult}_A^!\varpi+Q^{\alpha''*\alpha'},\  \\
 \varpi^{\alpha',\alpha''}&=\pr_1^!\varpi+\pr_2^!\varpi+Q^{\alpha'}+Q^{\alpha''}.
\end{split}\]
Using the property \eqref{Qconcat} of $Q^\alpha$ under concatenation,
we obtain:
\begin{proposition}\label{prop:mult}
  The 2-form $\varpi$ satisfies, over $A^{[2]}\subset A\times A$, 
\[\on{mult}_A^!\varpi=\pr_1^!\varpi+\pr_2^!\varpi-\lambda\]
Here $\lambda\in \Om^2(G\times G)$ is the 2-form,
$\lambda=\hh \pr_1^*\theta^L\cdot \pr_2^*\theta^R$.
\end{proposition}
This `lifts' the property of the Cartan 3-form, 
$\on{mult}_G^*\eta=\pr_1^*\eta+\pr_2^*\eta-\d \lambda$.

\section{Pull-backs}
\subsection{The lifting problem for $\Phi^!A$}
Given a $G$-equivariant map $\Phi\colon M\to G$, consider the pull-back algebroid
$A_M=\Phi^!A\to M$.  Sections of $A_M$ are pairs $(X,\xi)$, where
$X\in\mf{X}(M)$ and $\xi\in C^\infty(M\times \R,\g)$ such that for all
$t$,
\[ \xi_{t+1}=\Ad_\Phi \xi_t +\iota_X\Phi^*\theta^R.\]
The bracket between two such sections reads, 
\[ [(X,\xi),(Y,\zeta)]_{A_M}=([X,Y],\ -[\xi,\zeta]_\g
+X\zeta-Y\xi),\]
and the anchor map is $\a_M(X,\xi)=X$. The sections $x_{A_M}=\Phi^!
x_A\in \Gamma(A_M)$ are generators for the $G$-action on $A_M$. 

Suppose $\alpha_t\in\Om^1(G,\g)^G$ is a family of 1-forms as in
\eqref{eq:standardalpha}, with $Q^\alpha=0$, thus
$\varpi^\alpha=\varpi$ and $\eta^\alpha_G=\eta_G$.  
Let $\varpi_M=\Phi^!\varpi\in \Gamma(\wedge^2 A_M^*)$. 
%
%Let $\theta^\alpha$ be
%the resulting connection on $A$, and
%$\theta^\alpha_M=\Phi^!\theta^\alpha$ its pull-back to $A_M$.  
%
Suppose 
\[ \Phi^*\eta_G=-\d_G\om.\] 
for an invariant 2-form $\om$. As shown in Section \ref{subsec:equivariant}, this
gives an equivariant solution of the lifting problem for $A_M$, 
relative to the central extension $\wh{L}_M=\Phi^*\wh{L}\to L_M=\Phi^*L$. 
Since we are assuming $\om_G=\om$, this  solution will have the additional
property that $j_{A_M}(x_{A_M})$ are generators for the action on
$\hat{A}_M$.  Since $\d_G\varpi_M=a_M^*\Phi^*\eta_G$, the sum
\[ \mathsf{a}_M^*\om+\varpi_M\in \Gamma(\wedge^2 A_M^*) \]
is equivariantly closed. Let us compute its kernel. For the following
theorem, we assume that the inner product on $\g$ is non-degenerate.

\begin{theorem}
Suppose $\Phi\colon M\to G$ is a $G$-equivariant map, and
$\om\in\Om^2(M)$ is an invariant 2-form 
such that $\d_G\om=-\Phi^*\eta_G$.

At any point $m\in M$, the kernel of $\mathsf{a}_M^*\om+\varpi_M\in \Gamma(\wedge^2
A_M^*)$ admits a direct sum decomposition,
\begin{equation}\label{eq:kernel}  
\ker(\mathsf{a}_M^*\om+\varpi_M)=\g\oplus (\ker(\om)\cap\ker(\d\Phi)).
\end{equation}
Here elements $v\in T_mM\cap \ker(\d_m\Phi)\subset T_mM$ are embedded
in $\ker(\mathsf{a}_M^*\om+\varpi_M)\subset A_M\subset TM\oplus A$ as elements of
the form $(v,0)$, while $\g$ is embedded diagonally as generators for
the action, $x\mapsto (x_M,{x_A})$.  
\end{theorem}
\begin{proof}
By definition, the fiber of $\Phi^!A=A_M$ at $m\in M$ is the subspace 
of $T_mM\oplus A_{\Phi(m)}$, consisting of pairs $(v,\xi)$ such that 
$(d_m\Phi)(v)=\a(\xi)$. 

The property $\d_G(\mathsf{a}_M^*\om+\varpi_M)=0$ means in particular that elements
of the form $(x_M,\,x_A)$ are in the kernel of $\om+\varpi_M$.  On
the other hand, elements of the form $(v,0)$ with $v\in \ker(d_m\Phi)$
are contained in $A_M$, and they are in the kernel of
$\om+\varpi_M$ if and only if $v\in \ker(\om)$.
This proves the inclusion $\supseteq$ in \eqref{eq:kernel}. 

For the opposite inclusion, consider a general element $(w,\xi)\in
A_M\subset TM\oplus A$ in the kernel of $\mathsf{a}_M^*\om+\varpi_M$ at $m\in M$. 
We have $\iota_{(w,\xi)}\varpi_M=\Phi^!\iota_\xi\varpi$, where
$\iota_\xi\varpi$ is given by the calculation following
\eqref{eq:varpiformula}. We thus obtain the condition
\[ \mathsf{a}_M^*(\iota_w\om)
-\int_0^1\dot{\xi}\cdot\kappa_M-\Ad_g(\xi_0)\cdot\mathsf{a}_M^*\theta^R-\hh
v_\xi\cdot \mathsf{a}_M^*\theta^R=0,\]
where $\kappa_M=\Phi^!\kappa$. Taking a contraction with $\zeta\in
\ker(\mathsf{a}_M)\cong L_{\Phi(m)}$, we obtain 
\[ \int_0^1\dot{\xi}\cdot\zeta=0.\]
Since this is true for all $\zeta\in  L_{\Phi(m)}$, the non-degeneracy
of the inner product implies $\dot{\xi}=0$. Thus $\xi$ is a
constant path. Letting $x=-\xi\in\g$, it follows that $(v,0)$
with $v=w-x_M$ lies in the kernel. As seen above, this means that
$v\in \ker(\om)\cap \ker(\d\Phi)$.
\end{proof}

The conditions, $\d_G\om=-\Phi^*\eta_G$ and $\ker(\om)\cap
\ker(\d\Phi)=0$ are exactly the defining conditions for a
\emph{q-Hamiltonian $G$-space} \cite{al:mom}. \footnote{In
  \cite{al:mom}, the second condition was stated in the form
  $\ker(\om)=\{\xi_M|\, \Ad_{\Phi(m)}\xi=-\xi\}$. The equivalence to
  $\ker(\om)\cap \ker(\d\Phi)=0$ was observed independently by
  Bursztyn-Crainic \cite{bur:di} and Xu \cite{xu:mom}.} That is, for a q-Hamiltonian $G$-space the
kernel of $\mathsf{a}_M^*\om+\varpi_M\in \Gamma(\wedge^2 A_M^*)$ is
the action Lie algebroid for the $G$-action, embedded as the Lie
subalgebroid of $A_M$ spanned by the generators of the $G$-action
$x_{A_M}$.

\subsection{The subalgebroid $A'$ and its pull-back $\Phi^!A'$}
Let $A'\subset A$ be the $G$-invariant subalgebroid, consisting of
$\xi\in A$ with $\xi_0=0$. 
Then $A'$ is again a transitive Lie algebroid, and $A=\g\ltimes A'$,
where $\g$ is embedded by the generators of the $G$-action.
The Lie algebroid $A'$ may be viewed as the Atiyah algebroid of the principal
$L_eG$-bundle $P_eG\to G$, where the subscript indicates paths based
at the group unit $e$. In turn, $P_eG$ may be identified with the
space $\Om^1(S^1,\g)$ of connections on $S^1=\R/\Z$, where the
identification is given by the map $\gamma\mapsto
\gamma^{-1}\d\gamma$. (Conversely, $\gamma$ is recovered by parallel
transport.)  There is a natural projection $q\colon A\to A',\ 
\xi\mapsto\xi-\xi(0)$, with $\a(q(\xi))=\a(\xi)+\xi(0)_G$. Of course,
$q$ does not preserve brackets.

Suppose now that $\Phi\colon M\to G$ is a $G$-equivariant map, and let
$A_M'=\Phi^!A'$.  The projection $q$ induces a projection map
$q_M\colon A_M\to A_M'$, given on sections by
\[ q_M(\xi,X)=(\xi-\xi(0),\ X+\xi(0)_M).\]
Its kernel is the trivial bundle $\g_M=M\times \g\subset A_M$, embedded by
the  map $x\mapsto (-x,x_M)$ generating the $G$-action. Even though
$q_M$ does not preserve brackets, we have: 
\begin{corollary}
If $E\subset A_M$ is a $G$-invariant Lie subalgebroid, transverse to
$\g_M$, then $q_M(E)\subset A'_M$ is a Lie subalgebroid of $A'_M$.  
\end{corollary} 
\begin{proof}
  The transversality implies that $q_M(E)$ is a sub-bundle of $A_M'$,
  of the same rank as $E$. Letting $\g_M=M\times\g\subset A_M$ be the
  embedding given by the generators of the $\g$-action, we have
\[ q_M(E)=q_M(E+\g_M)=(E+\g_M)\cap A_M'.\]
But the sections of $E+\g_M$ are closed under $[\cdot,\cdot]_A$, as
are the sections of $A_M'$.
\end{proof}

\begin{example}
In this example, we assume that $G$ is compact and that the inner
product $\cdot$ on $\g$ is positive definite.  
Let $\Phi\colon \Co\subset G$ be the inclusion of a conjugacy
class. Then $A_\Co=\Phi^!A$ is a sum 
\[ A_\Co= L + \g_\Co.\]
The intersection $L\cap \g_\Co$ is the sub-bundle of $\g_\Co$, spanned
by $(-x,x_\Co|_g)$ with $x\in \ker(\Ad_g-1)$. Let the fibers $L_g$ carry the inner product defined by the integration pairing. Then, after
an appropriate Hilbert space completion (for instance, using the
Sobolev space $W^{1,2}$), we obtain
\[ L_g^\C=L_g^+\oplus \ker(\Ad_g-1)^\C \oplus L_g^-\] 
where $L_g^\pm$ are the direct sum of the eigenspaces for the
positive/negative part of the spectrum of $\f{1}{\sqrt{-1}}\f{\p}{\p
  t}$, and $\ker(\Ad_g-1)^\C \cong L_g^0$ is embedded as the kernel. 
Consequently, 
\[ A_\Co^\C=\Phi^!A= L^+\oplus \g_\Co\oplus L^-.\]
Since $L^\pm$ are Lie algebra sub-bundles of $L$, their integrability is
automatic, and hence 
\[ (A'_\Co)^\C=q(L^+)\oplus q(L^-)\]
is an integrable polarization of $A'_\Co$. Letting $\O$ be the
coadjoint $LG$-group orbit corresponding to $\Co=\O/L_0G$, the bundle 
$\A'_\Co$ is interpreted as $T\O/L_0G$, and its polarization is the
standard K\"ahler structure. 

\end{example}

\section{Higher analogues of the form $\varpi$}

We had remarked above that the Cartan form $\eta$ may be viewed as a
Chern-Simons form, and similarly $\eta_G$ as an equivariant
Chern-Simons form. For any invariant polynomial $p\in (S^m\g^*)^G$, we
may define `higher analogues' $\eta^p,\ \eta^p_G$ of the Cartan form
using the theory of Bott forms. We will not assume the existence of an
invariant inner product on $\g$.
\subsection{Bott forms}\label{subsec:bott1}
Let $N$ be a manifold.  Suppose $\beta\in\Om^1(N,\g)$, and that $p\in
(S^m\g^*)^G$ is an invariant polynomial of degree $m$. Then $p(F^\beta)$ is closed, as
an application of the Bianchi identity $\d F^\beta+[\beta,F^\beta]=0$.
\footnote{For any polynomial $p\in S\g^*$, we define its derivative
  $p'\in S\g^*\otimes \g^*$ by $\l p'(v),w\r=\f{\p}{\p t}\Big|_{t=0}
  p(v+tw)$.  If $p$ is $G$-invariant, then $[x,y]\cdot p'(y)=0$ for
  all $x,y\in\g$.  Thus $\d p(F^\beta)=\d F^\beta\cdot
  p'(F^\beta)= - [\beta,F^\beta]\cdot p'(F^\beta)=0$. }
Given $\beta_0,\ldots,\beta_k\in\Om^1(N,\g)$ we
define \emph{Bott forms} $\Upsilon^p(\beta_0,\ldots,\beta_k)\in \Om^{2m-k}(N)$
\[ \Upsilon^p(\beta_0,\ldots,\beta_k)=(-1)^{[\f{k+1}{2}]}\int_{\Delta^k} p(F^\beta).\]
Here $\Delta^k=\{s\in \R^{k+1}|\ s_i\ge 0,\ \ \sum_{i=0}^k s_i=1\}$
is the standard $k$-simplex, and $\beta=\sum_{i=0}^k s_i \beta_i$,
viewed as a form $\beta\in \Om^1(N\times \Delta^k,\g)$. For a
detailed discussion of Bott forms, see \cite[Chapter 4]{va:sym}.
The Bott forms satisfy
\[ \begin{split}
\d \Upsilon^p(\beta_0,\ldots,\beta_k)&=\sum_{i=0}^k (-1)^i 
\Upsilon^p(\beta_0,\ldots,\wh{\beta}_i,\ldots \beta_k),
\\
\Upsilon^p(\Phi\bullet\beta_0,\ldots,\Phi\bullet\beta_k)&=
\Upsilon^p(\beta_0,\ldots,\beta_k),\ \ \Phi\in C^\infty(N,G).
\end{split}\]
The first identity follows from Stokes' theorem \cite[Theorem
4.1.6]{va:sym}, while the second identity comes from the gauge 
equivariance of the curvature, $F^{\Phi\bullet\beta}=\Ad_\Phi(F^\beta)$. 

Consider the special case $N=G$. For any $p\in (S^m\g^*)^G$ we define 
\[ \eta^p=\Upsilon^p(0,\theta^L)\in\Om^{2m-1}(G).\]
Then $\d\eta^p=\Upsilon^p(\theta^L)-\Upsilon^p(0)=0$, using that
$F^{\beta}=0$ for both $\beta=0,\theta^L$.
For $G$ compact, the classes $[\eta^p]$ are known to generate the
cohomology ring $H^*(G,\R)$.

\subsection{Equivariant Bott forms}
With small modifications, the definition of  Bott forms carries
over to equivariant 1-forms $\beta\in \Om^1(N,\g)^G$ for a given
$G$-action on $N$, and for the adjoint action of $G$ on $\g$. 
For any such form, and an invariant polynomial $p$,
the equivariant Bianchi identity $\d_G F_G^\beta+[\beta,F_G^\beta(x)+x]=0$
implies that $p(F_G^\beta(x)+x)$ is equivariantly closed.
Given $\beta_0,\ldots,\beta_k\in \Om^1(N,\g)^G$ we define
equivariant Bott forms $\Upsilon^p_G(\beta_0,\ldots,\beta_k)\in
\Om_G(N)$ by
\[  \Upsilon^p_G(\beta_0,\ldots,\beta_k)(x)=(-1)^{[\f{k+1}{2}]}\int_{\Delta^k} p(F_G^\beta(x)+x),\]
with $\beta=\sum_{i=0}^k s_i \beta_i$ as above.
Then 
\[ \begin{split}
\d_G \Upsilon^p_G(\beta_0,\ldots,\beta_k)&=\sum_{i=0}^k (-1)^i 
\Upsilon^p_G(\beta_0,\ldots,\wh{\beta}_i,\ldots \beta_k),
\\
\Upsilon^p_G(\Phi\bullet\beta_0,\ldots,\Phi\bullet\beta_k)&=
\Upsilon^p_G(\beta_0,\ldots,\beta_k),\ \ \ \ \Phi\in C^\infty(N,G)^G.
\end{split}\]
Again this follows from Stokes' theorem, respectively from the
property $F_G^{\Phi\bullet\beta_i}(x)+x=\Ad_\Phi(F_G^{\beta_i}(x)+x)$ of
the equivariant curvature.

If $N=G$ with conjugation action, and $p\in (S^m\g^*)^G$ we define \cite{je:gr}
\[ \eta^p_G=\Upsilon^p_G(0,\theta^L)\in \Om^{2m-1}_G(G).\]
Since $F^{\theta^L}_G(x)+x=\Ad_{g^{-1}}(x)$, we have
\[ \d_G\eta^p_G(x)=\Upsilon^p_G(\theta^L)-\Upsilon^p_G(0)
=p(\Ad_{g^{-1}}(x))-p(x)=0.\]
Thus $\eta^p_G$ are closed equivariant extensions of $\eta^p$.

\subsection{Families of flat connections}
Suppose that $\beta_t\in\Om^1(N,\g)^G$ is a family of invariant
1-forms, such that $F_G^{\beta_t}(x)+x=0$ for all $t$.  Then
\[ \d_G \Upsilon_G^p(0,\beta_t)(x)=-p(x)\] 
for all $t$, and so the difference
$\Upsilon_G^p(0,\beta_t)-\Upsilon_G^p(0,\beta_0)$ is equivariantly
closed. We will construct an equivariant primitive.  Let $\beta\in
\Om^1(N\times \Delta^1\times I,\g)^G$ be given as
\[\beta_{s,t}=s\beta_t,\ \ \ t\in I=[0,1],\ \ s\in
\Delta^1\cong [0,1].\]
We set 
\[ I^p_G(\{\beta_t\})(x)=\int_{\Delta^1\times I}
p(F^\beta_G(x)+x).\] 
\begin{lemma}\label{lem:upsilon}
If $m=\deg(p)\ge 2$,
\[{\Upsilon}^p_G(0,\beta_1)-{\Upsilon}^p_G(0,\beta_0)=\d_G
I^p_G(\{\beta_t\}),\]
\end{lemma} 
\begin{proof}
We compute $\d_G I^p_G(\{\beta_t\})(x)$ by Stokes' theorem. There will
be four boundary contributions, corresponding to the four sides $s=0,\
s=1,\ t=0,\ t=1$ of the square $\Delta^1\times I$. The boundary
contribution for $s=1$ is given as the integral of 
\[ p(\d t\wedge \dot{\beta}_{t}+F_G^{\beta_{t}}(x)+x).\]

But $F_G^{\beta_{t}}(x)+x=0$ by assumption, and hence $p(\d t\wedge
\dot{\beta}_{t})=0$ since $\deg(p)\ge 2$. The boundary contribution of
$s=0$ vanishes as well, since the pull-back of $p(F^\beta_G(x)+x)$ has
no $\d t$-component there. The remaining two boundary contributions 
are ${\Upsilon}^p_G(0,\beta_1)$ and $-{\Upsilon}^p_G(0,\beta_0)$ as
desired. 
\end{proof}

The discussion for the non-equivariant case is essentially the same:
Given a family $\beta_t\in\Om^1(N,\g)$ with $F^{\beta_t}=0$, the
integral $I^p(\{\beta_t\})=\int_{\Delta^1\times I} p(F^\beta)$ has the
property $\Upsilon^p(0,\beta_1)-\Upsilon^p(0,\beta_0)=\d
I^p(\{\beta_t\})$. Writing $F^\beta=\d s \wedge\beta_t+s\d t
\wedge\dot{\beta}_t+\f{s(s-1)}{2}[\beta_t,\beta_t]$, we may carry out
the $s$-integration in the definition of $\Upsilon^p$, and find 
that $\Upsilon^p$ is explicitly given as a rational multiple of 
\begin{equation}\label{eq:explicit}
 \int_0^1
p(\beta_t,\,\dot{\beta}_t,\,[\beta_t,\beta_t],\,\ldots,\,[\beta_t,\beta_t]).\end{equation}
Here we have associated to $p\in (S^m\g^*)^G$ the multilinear form
(again denoted $p$) such that $p(x,\ldots,x)=p(x)$.

\subsection{The form $\varpi^p_G$}
The theory described above works equally well for $\Om(N)$ replaced
with $\Gamma(A)$, for $A\to N$ a Lie algebroid. In the $G$-equivariant
case, one has to require that the $G$-action on $A$ admits
infinitesimal generators $x_A$. As before, we will view 
$\Om(N)\subset \Gamma(\wedge A^*)$ respectively 
$\Om_G(N)\subset \Gamma_G(\wedge A^*)$ as the basic subcomplexes. 

Our goal is to construct primitives of $\a^*\eta^p_G\in
\Gamma_G(\wedge A^*)$, where $A\to G$ is the Atiyah algebroid over
$G$. Let $\kappa_t\in \Gamma(A^*)\otimes\g$ as in Section
\ref{subsec:kappat}. With $I^p_G(\{\kappa_t\})\in \Gamma_G(\wedge
A^*)$ as above, put
\[ \varpi^p_G=I^p_G(\{\kappa_t\})-\Upsilon^p_G(0,\a^*\theta^L,\kappa_0).\]
\begin{theorem}
The forms $\varpi^p_G$ are equivariant primitives of $\a^*\eta^p_G$:
\[ \d_G\varpi^p_G(x)=\a^*\eta^p_G(x).\]
\end{theorem}
\begin{proof}
Since $\kappa_{1}=g\bullet\kappa_0$ by Lemma \ref{lem:kappa}, we have 
\[ \Upsilon^p_G(0,\kappa_1)=\Upsilon^p_G(0,g\bullet\kappa_0)
=\Upsilon^p_G(g^{-1}\bullet 0,\kappa_0)=\Upsilon^p_G(\a^*\theta^L,\kappa_0).\]
Lemma \ref{lem:kappa} also shows that $F_G^{\kappa_t}(x)+x=0$. Hence
Lemma 
\ref{lem:upsilon} applies and
gives
\[\begin{split}
\d_G I^p_G(\{\kappa_t\})&=
{\Upsilon}^p_G(0,\kappa_1)-{\Upsilon}^p_G(0,\kappa_0)\\&=
\Upsilon^p_G(\a^*\theta^L,\kappa_0)+{\Upsilon}^p_G(\kappa_0,0)\\
&=\Upsilon^p_G(\a^*\theta^L,0)+\d_G\Upsilon^p_G(0,\a^*\theta^L,\kappa_0).
\qedhere\end{split}
\]
\end{proof}

Setting the equivariant parameter equal to $0$, i.e. defining
$\varpi^p=\varpi^p_G(0)$, this also gives in
particular non-equivariant primitives, $\d\varpi^p=\a^*\eta^p$. 

\subsection{The case $p(x)=\hh x\cdot x$}
If $p$ is homogeneous of degree $\deg(p)=2$, the formulas simplify.
With $\beta_{s,t}=s\kappa_t$, the definition of
$I^p_G(\{\kappa_t\}(x)$ gives
\[ I^p_G(\{\kappa_t\})(x)=\int_{\Delta^1\times I}p(F_G^\beta(x)+x)
=\int_{\Delta^1\times I}p(\d s\wedge \kappa_t+s \d
t\wedge\dot{\kappa}_t).\]
Indeed, only the coefficient of $\d s\wedge \d t$ in
$p(F_G^\beta(x)+x)$ will contributes to the integral.  Hence
\[ I^p_G(\{\kappa_t\})(x)=\int_0^1 p(\kappa_t,\dot{\kappa}_t),\]
where we associated to $p$ a symmetric bilinear form, again denote
$p$, with $p(x,x)=p(x)$. In particular, $I^p_G(\{\kappa_t\})=
I^p(\{\kappa_t\})$. 
A similar discussion applies to the 2-dimensional integral defining
$\Upsilon^p_G(0,\a^*\theta^L,\kappa_0)$. One obtains
\[ \Upsilon^p_G(0,\a^*\theta^L,\kappa_0)(x)=p(\a^*\theta^L,\kappa_0),\]
which again is independent of $x$. We conclude that if $p(x)=\hh x\cdot x$
for an invariant inner product $\cdot$ on $\g$, then $\varpi_G^p$
coincides with $\varpi^p$, and is given by the Formula
\eqref{eq:varpiformula}.

\subsection{Pull-back to the group unit}
The inclusion map $\iota\colon \{e\}\to G$ is $G$-equivariant, and
lifts to a morphism of Lie algebroids, $L\g\to A$. (In fact, 
$L\g=\iota^! A$.) 
Let 
\[ \sig^p=\iota^!\varpi^p,\ \ \sig^p_G=\iota^!\varpi^p_G\]
be the resulting elements of $\Gamma(\wedge L\g^*)$, resp.  
$\Gamma_G(\wedge L\g^*)$. Since $\iota^*\eta^p_G=0$, it is immediate
that these forms are closed (resp. equivariantly closed) for the Lie
algebra differential.

The pull-back of $\kappa^{L\g}:=\iota^!\kappa$ may be viewed as minus the
right-invariant Maurer-Cartan forms for the group $LG$. 
Since the pull-back of $\Upsilon^p(0,\a^*\theta^L,\kappa_0)$ vanishes,
Equation \eqref{eq:explicit} shows that $\sig^p$ is a rational
multiple of 
\[
p(\kappa_t^{L\g},\,\dot{\kappa}_t^{L\g},\,[\kappa_t^{L\g},\,\kappa_t^{L\g}],\ldots,\,
[\kappa_t^{L\g},\,\kappa_t^{L\g}]).\]
These forms are discussed by Pressley-Segal in \cite[Chapter
4.11]{pr:lo}, who prove that for compact $G$ the cohomology ring
$H^*(LG)$ is generated by the left-invariant forms, and is in fact
isomorphic to the Lie algebra cohomology of $L\g$. The forms $\sig^p$
arise as some of the generators of the cohomology. (The remaining
generators are obtained by pull-back under the evaluation map $LG\to
G,\ \gamma\mapsto \gamma_0$). Our theory thus provides closed
$G$-equivariant extensions of the Pressley-Segal generators, and gives
an explicit transgressions of these forms to $\eta^p,\eta^p_G$.

\begin{appendix}
\section{Chern-Simons forms on Lie algebroids}
In this appendix, we extend some formulas for Chern-Simons forms to
the case of Lie algebroids. We omit proofs, which are all given by
straightforward calculations (extending the well-known case $A=TN$). 

\subsection{Non-equivariant Chern-Simons forms}\label{subsec:CS1}
Suppose $A\to N$ is a Lie algebroid. We will consider the elements of
$\Gamma(\wedge A^*)$ as `forms on $A$'.  For any $\g$-valued 1-form
$\beta\in \Gamma(A^*)\otimes\g$ with `curvature'
$F^\beta=\d\beta+\hh[\beta,\beta]_\g$, the 4-form $\hh F^\beta\cdot
F^\beta\in \Gamma(\wedge^4 A^*)$ is exact. A primitive is given by the
\emph{Chern-Simons form} $\on{CS}(\beta)=\Upsilon^p(0,\beta)$ for
$p(x)=\hh x\cdot x$, where we have used the notation from Section
\ref{subsec:bott1}. Thus $\d\on{CS}(\beta)=p(F^\beta)=\hh F^\beta\cdot
F^\beta$.  A short calculation gives
\[\on{CS}(\beta)=\f{1}{2}(\d\beta)\cdot\beta+\f{1}{6}\beta\cdot[\beta,\beta]_\g\in
\Gamma(\wedge^3 A^*).\]
For $\Phi\in C^\infty(N,G)$ let
$\Phi\bullet\beta=\Ad_\Phi(\beta)-\Phi^*\theta^R$ be the gauge
transform of $\beta$. (Here the last term is viewed as an element of
of $\Gamma(A^*)$, by the pull-back map $\Om(N)\to \Gamma(\wedge A^*)$.

\begin{proposition}
For $\beta\in \Gamma(A^*)\otimes \g$ and $\Phi\in C^\infty(N,G)$, we have 
\begin{equation}
 \on{CS}(\Phi\bullet\beta)=\on{CS}(\beta)+\Phi^*\eta-\hh \d\,
 (\beta\cdot \Phi^*\theta^L).
\end{equation}
Given a smooth family $\beta_t$
one has the transgression formula,
\begin{equation}\label{eq:transgression}
\f{\p}{\p t}\on{CS}(\beta_t)=\dot{\beta}_t\cdot F^{\beta_t}-\hh
\d\,(\beta_t \cdot\dot{\beta}_t).
\end{equation}
\end{proposition}
Suppose $\beta_{t+1}=\Phi\bullet \beta_t$ for some given gauge
transformation $\Phi\in C^\infty(N,G)$. Integrating
\eqref{eq:transgression} over $[0,1]$, and using the property 
of Chern-Simons forms under gauge transformations, we obtain
\begin{equation} \label{eq:CSform}
\int_0^1 \dot{\beta}_t\cdot F^{\beta_t}
=\Phi^*\eta+\d Q^\beta\end{equation}
where $Q^\beta\in\Gamma(\wedge^2 A^*)$ is the 2-form, 
\[ Q^\beta=\hh \Phi^*\theta^L\cdot\beta_0+\hh \int_0^1 \beta_t
\cdot\dot{\beta}_t.\]

\subsection{$G$-equivariant Chern-Simons forms}
Suppose that the group $G$ acts on $A\to N$, with infinitesimal
generators $x\mapsto x_A$. Then we can consider the complex 
$\Gamma_G(\wedge A^*)$ of $G$-equivariant forms. 

Suppose $\beta\in (\Gamma(A^*)\otimes\g)^G$, and let 
$F_G^\beta=\d_G\beta+\hh[\beta,\beta]$ be its `equivariant
curvature'. We have 
\[ \d_G F_G^\beta+[\beta,F_G^\beta(x)+x]=0.\]
As a consequence, the equivariant 4-form $p(F_G^\beta(x)+x)-p(x)$ for
$p(x)=\hh x\cdot x$ is equivariantly closed. 
\footnote{In the case $A=TN$, the form $\beta$ may be regarded 
as the restriction to $N\times\{e\}$ of a principal connection on
$N\times G$, invariant relative to the diagonal
action $k.(n,u)=(k.n,ku)$. The pull-back of the $G$-equivariant
curvature $F_G^\theta(x)$ to $N\times\{e\}$ is $F_G^\beta(x)+x$.}
Let $\on{CS}_G(\beta)=\Upsilon^p_G(0,\beta)
\in\Gamma_G(\wedge^3 A^*)$, with differential
$p(F_G^\beta(x)+x)-p(x)$. One finds
\[ \on{CS}_G(\beta)(x)=\f{1}{2} 
\d_G\beta(x)\cdot\beta +\f{1}{6}\beta\cdot[\beta,\beta]_\g+\beta\cdot
x.\]

\begin{proposition}
For $\beta\in (\Gamma(A^*)\otimes\g)^G$ and $\Phi\in C^\infty(N,G)^G$,  
\begin{equation} 
\on{CS}_G(\Phi\bullet\beta)=\on{CS}_G(\beta)+\Phi^*\eta_G-\hh\d_G(\beta\cdot\Phi^*\theta^L).
\end{equation}
Given a smooth family $\beta_t\in (\Gamma(A^*)\otimes\g)^G$, one has 
\[
\f{\p}{\p t}\on{CS}_G(\beta_t)(x)=\dot{\beta}_t\cdot (F^{\beta_t}_G(x)+x)-\hh
\d\,(\beta_t \cdot\dot{\beta}_t).
\]
\end{proposition}
Hence, if $\beta_t\in (\Gamma(A^*)\otimes\g)^G$ is a family of invariant forms
with $\beta_1=\Phi\bullet \beta_0$, and letting $Q^\beta$ be defined as
above, one finds
\begin{equation}\label{eq:CSform2}
 \int_0^1 \dot{\beta}_t\cdot (F^{\beta_t}_G(x)+x)=\Phi^*\eta_G+\d_G Q^\beta.\end{equation}

\subsection{Properties of the functional $Q$}
Here are some properties of the functional $Q(\beta)=Q^\beta$.
\begin{proposition}[Properties of the functional $Q$]\label{prop:q}
\noindent
\begin{enumerate}
\item {\bf Reparametrization invariance.}
Let $\beta_t \in 
\Gamma(A^*)\otimes\g$ be a smooth family of forms with $\beta_{t+1}=\Phi\bullet
\beta_t$, and suppose $\phi\colon \R\to \R$ is an orientation preserving
  diffeomorphism such that $\phi(t+1)=\phi(t)+1$. Then 
$Q(\beta\circ
  \phi)=Q(\beta)$. 
%Similarly, for an orientation reversing
%  diffeomorphism with $f(t+1)=f(t)-1$ one has $Q(\beta\circ
%  f)=-Q(\beta)$. 
\item {\bf Multiplicative property.} 
Let $\beta',\beta''\colon \R\to
\Gamma(A^*)\otimes\g $ be two maps such that $\beta'_{t+1}=\Phi'\bullet\beta'_t$,
$\beta''_{t+1}=\Phi''\bullet\beta''_t$. Suppose
$\beta'_1=\beta''_0$, and let the concatenation 
be defined for $0\le t\le 1$ by 
\[  (\beta''*\beta')_t=\begin{cases}\beta'_{2t}& 0\le
  t\le \hh \\
\beta''_{2t-1}& \hh \le t\le 1.
\end{cases}\]
and extend to all $t$ by the property, $
(\beta''*\beta')_{t+1}=(\Phi''\Phi')\bullet
(\beta''*\beta')_t$.  (The resulting $\beta$ is piecewise smooth, and it is
smooth e.g. if $\beta',\beta''$ are constant near $t=0$.) Then
\begin{equation}\label{Qconcat}
Q(\beta''*\beta')=Q(\beta')+Q(\beta'')+(\Phi',\Phi'')^*\lambda \end{equation}
where $\lambda\in\Om^2(G\times G)$ is the 2-form, 
$\lambda=\hh \pr_1^*\theta^L\cdot \pr_2^*\theta^R$. 
\item {\bf Inversion.}
Let $\beta\colon \R\to \Om^1(N,\g)$ with $\beta_{t+1}=\Phi\bullet
\beta_t$, 
and define $\beta^-_t=\beta_{-t}$. Then 
$\beta^-_{t+1}=\Phi^{-1}\bullet\beta_t^-$, and we have
$Q(\beta^-)=-Q(\beta)$.  
\end{enumerate}
\end{proposition}
\begin{proof}
(a)  
%Suppose first that $f$ is orientation preserving. 
The claim is obvious if $\phi(0)=0$, since both the integral and the
  term $\hh \Phi^*\theta^L\cdot\beta_0$ are unchanged in this case. It
  remains to check the case $\phi(t)=t+u$, for some fixed $u\in\R$. It is
  enough to consider the case $0\le u\le 1$. We have
\[ \begin{split}\int_0^1 \beta_{t+u}\cdot \dot{\beta}_{t+u}&=\int_u^{1+u} \beta_t
  \cdot\dot{\beta}_t \\
&=\int_u^1 \beta_t
  \cdot\dot{\beta}_t+\int_0^u (\Ad_\Phi\beta_{t}-\Phi^*\theta^R)\cdot
  \Ad_\Phi\dot{\beta}_{t}\\
&=\int_0^1 \beta_t
  \cdot\dot{\beta}_t-\int_0^u \Phi^*\theta^L\cdot \dot{\beta}_t\\
&= \int_0^1 \beta_t \cdot\dot{\beta}_t-\Phi^*\theta^L\cdot (\beta_u-\beta_0).\qedhere
\end{split}\]
%For the orientation reversing case, one only needs to check the case
%$f(t)=1-t$. 

(b) In calculating $Q(\beta)-Q(\beta')-Q(\beta'')$, the integral
contributions cancel out, and we are left with 
\[ \begin{split}
Q(\beta)-Q(\beta')-Q(\beta'')
&=\hh\big((\Phi''\Phi')^*\theta^L\cdot \beta_0-(\Phi')^*\theta^L\cdot\beta_0
-(\Phi'')^*\theta^L\cdot \beta_{1/2}\big).
\end{split}\]
Since $\beta_{1/2}=\beta'_t=\Phi'\bullet\beta_0
=\Ad_{\Phi'}\beta_0-(\Phi')^*\theta^R$ and
$(\Phi''\Phi')^*\theta^L=(\Phi')^*\theta^L+\Ad_{(\Phi')^{-1}}(\Phi'')^*\theta^L$,
we are left with $\hh (\Phi')^*\theta^L\cdot (\Phi'')^*\theta^R$.

(c) is a straightforward calculation. 
\end{proof}
\end{appendix}

\def\cprime{$'$} \def\polhk#1{\setbox0=\hbox{#1}{\ooalign{\hidewidth
  \lower1.5ex\hbox{`}\hidewidth\crcr\unhbox0}}} \def\cprime{$'$}
  \def\cprime{$'$} \def\polhk#1{\setbox0=\hbox{#1}{\ooalign{\hidewidth
  \lower1.5ex\hbox{`}\hidewidth\crcr\unhbox0}}} \def\cprime{$'$}
  \def\cprime{$'$}
\providecommand{\bysame}{\leavevmode\hbox to3em{\hrulefill}\thinspace}
\providecommand{\MR}{\relax\ifhmode\unskip\space\fi MR }
\providecommand{\MRhref}[2]{%
  \href{http://www.ams.org/mathscinet-getitem?mr=#1}{#2}
}
\providecommand{\href}[2]{#2}

\end{document}